\documentclass[12pt]{amsart}
\usepackage{amsfonts}
\usepackage{fullpage}
\usepackage{amsmath,amssymb,amsfonts,amsthm}
\usepackage{amsmath,amsthm,bbm}
\usepackage{latexsym}
\usepackage{array}
\usepackage{times}
\usepackage{dsfont}
\usepackage{amssymb}
\usepackage{graphics}
\usepackage{enumerate}
\usepackage[english]{babel}
\usepackage{color}
\definecolor{marin}{rgb}   {0.,   0.3,   0.7} 
\definecolor{rouge}{rgb}   {0.8,   0.,   0.} 
\definecolor{sepia}{rgb}   {0.8,   0.5,   0.} 
\usepackage[colorlinks,citecolor=marin,linkcolor=rouge,
            bookmarksopen,
            bookmarksnumbered
           ]{hyperref}

\newtheorem{lemma}{Lemma}[section]

\newtheorem{theorem}[lemma]{Theorem}
\newtheorem{proposition}[lemma]{Proposition}

\newtheorem{remark}[lemma]{Remark}
\newtheorem{example}[lemma]{Example}

\newtheorem{notation}[lemma]{Notation}
\newtheorem{definition}[lemma]{Definition}
\newtheorem{conclusion}[lemma]{Conclusion}
\numberwithin{equation}{section}
\newcommand{\QED}{\mbox{}\hfill \raisebox{-0.2pt}{\rule{5.6pt}{6pt}\rule{0pt}{0pt}} 
          \medskip\par}

\newcommand{\eps}{\varepsilon}

\newcommand{\dd}{\mathrm{d}}

\newcommand{\gs}{\mathfrak{g}}

\newcommand{\Hc}{\mathcal{H}}

\newcommand{\R}{\mathbb{R}}

\newcommand{\T}{\mathbb{T}}
\newcommand{\Lc}{\mathcal{L}}

\newcommand{\Z}{\mathbb{Z}}

\newcommand{\ssf}{\mathsf{s}}

\newcommand{\zs}{\mathfrak{Z}}

\newcommand{\Norm}[2]{\|#1\|\left.\vphantom{T_{j_0}^0}\!\!\right._{#2}}

\author[E. Faou]{Erwan Faou}
\address{INRIA-Rennes Bretagne Atlantique and IRMAR (UMR 6625) Universit\'e de Rennes I} 
\email{Erwan.Faou@inria.fr}

\author[R. Horsin]{Romain Horsin}
\address{INRIA-Rennes Bretagne Atlantique and IRMAR (UMR 6625) Universit\'e de Rennes I} 
\email{Romain.Horsin@inria.fr}

\author[F. Rousset]{ Fr\'ed\'eric Rousset }
\address{ Laboratoire de Math\'ematiques d'Orsay (UMR 8628) Universit\'e Paris-Sud et Institut Universitaire de France}
  \email{frederic.rousset@math.u-psud.fr}
\thanks{This work is partially supported by the ERC starting grant GEOPARDI No. 279389}

\title[Numerical Landau damping]
{ On numerical Landau damping for splitting methods applied to the Vlasov-HMF model }

\begin{document}

\begin{abstract}
We consider time discretizations of the Vlasov-HMF (Hamiltonian Mean-Field) equation based on splitting methods between the linear and non-linear parts. 
We consider solutions starting in a small Sobolev neighborhood of a  spatially homogeneous state  satisfying
 a linearized stability criterion (Penrose criterion). 
We prove that   the numerical solutions exhibit  a scattering behavior to a modified state, which implies a nonlinear Landau damping effect with polynomial rate of damping. Moreover, we prove that the modified state is close to the continuous one and provide error estimates with respect to the time stepsize. 
\end{abstract}

\thanks{
}

\maketitle
%\tableofcontents

\section{Introduction}

In this paper we consider time discretizations of the Vlasov-HMF model. This model has received much interest in the physics and mathematics
litterature for many reasons:  It  is a simple ideal toy model that keeps several features of the long range interactions,  it is 
a simplification of physical systems like charged or gravitational sheet models and it is rather easy to make numerical simulations on it. We refer for example  to \cite{Bouchet1}, \cite{Bouchet2}, \cite{Bouchet3}, \cite{Caglioti-Rousset1}, \cite{Caglioti-Rousset2}   for more details. 
This model also has strong analogy with the Kuramoto model of coupled oscillators in its  continuous limit \cite{Gerard-Varet}, \cite{Caglioti}, \cite{Dietert}. 
A  long time analysis of the Vlasov-HMF  model around homogenous stationary states has been recently  performed in \cite{RF} where a Landau damping result is proved in Sobolev regularity. The purpose of the present paper is in essence to show that this result persists through time discretization by splitting methods.  

 The Vlasov-HMF model reads
\begin{equation}
\label{eq:hmf1}
\partial_t f (t,x,v)+  v \partial_x f (t,x,v) =  \partial_{x}
 \Big(\int_{\R \times \mathbb{T}} P(x-y) f(t,y,u) \dd u \dd y \Big) \partial_v f(t,x,v), 
\end{equation}
%{\color{red} je crois qu'il manquait un $\partial_{x}$}
where $(x,v) \in \T \times \R$ and  the kernel $P(x)$ is given by  $P(x) = \cos(x)$. We consider  initial data under the form $ f_{0}(x,v)= \eta(v) + \eps r_{0}(x,v)$
 where $\eps$ is a small parameter and $r_{0}$ is of size one (in a suitable functional space). This means that we study small perturbations of a stationary solution $\eta(v)$.
Writing the exact solution as 
$$
f(t,x,v) = \eta(v) + \varepsilon r(t,x,v), 
$$
and setting 
\begin{equation}
\label{cimarosa}
g(t,x,v) = r(t,x + tv,v),
\end{equation}
the main result given in \cite{RF} is that if $\varepsilon$ is small enough,  $g(t,x,v)$ converges towards some  $g^\infty(x,v)$ when $t$ goes to $ \infty$ in Sobolev regularity. This results implies  a Landau damping phenomenon for the solution. 

In this paper, we consider the time discretization of \eqref{eq:hmf1} by splitting methods based on the decomposition of the equation between the free part
\begin{equation}
\label{transport}
\partial_t f (t,x,v)+  v \partial_x f (t,x,v) = 0, \quad f(0,x,v) = f^0(x,v),
\end{equation}
whose solution is given explicitely by $\varphi_T^t(f^0)(x,v) := f^0(x - tv,v)$, and the potential part 
\begin{equation}
\label{boston}
\partial_t f (t,x,v) =  \partial_{x}
 \Big(\int_{\R \times \mathbb{T}} P(x-y) f(t,y,u) \dd u \dd y \Big) \partial_v f(t,x,v), \quad f(0,x,v) = f^0(x,v),
\end{equation}
whose solution is explicitely given by 
$$
\varphi_P^t(f^0) = f^0(x,v + t E(f^0,x)),
$$
where $E(f,x) = \partial_{x}
 \Big(\int_{\R \times \mathbb{T}} P(x-y) f(y,u) \dd u \dd y \Big)$ is indeed kept constant during the evolution of \eqref{boston}. 

 The Lie splittings we consider are given by the formulas
\begin{equation}
\label{lie1}
f^{n+1}=\varphi_{P}^{h} \circ \varphi_{T}^{h}(f^{n}),
\quad
\mbox{or}
\quad
f^{n+1}=\varphi_{T}^{h} \circ \varphi_{P}^{h}(f^{n}),
\end{equation}
where $h >0$ is the time step. 
The functions $f^n(x,v)$ defined above are  {\em a priori} order one approximations of $f(t,x,v)$ at time $t = nh$. 

We  also consider the Strang splitting 
 \begin{equation}
 \label{strang}
 f^{n+1} = \varphi_T^{h /2} \circ \varphi_P^h  \circ \varphi_T^{h /2} (f^n)
 \end{equation}
that should provide an order two approximation $f^{n}(x,v)$ of $f(t,x,v)$ at time $t=nh$ (the same being expected for the symmetric splitting where the roles of $T$ and $P$ are swapped).\\

We can then define the sequence of function $r^n(x,v)$ by the formula
\begin{equation}
\label{ravel}
f^n(x,v) = \eta(v) + \varepsilon r^n(x,v), 
\end{equation}
and the functions 
\begin{equation}
\label{cimarosan}
g^n(x,v) = r^n(x + n h  v,v)
\end{equation} 
which have to be thought as approximations of $g(t,x,v)$ at time $t = n h $. 

The main result of our paper is that if $\varepsilon$ and $h $ are small enough, $g^n(x,v)$ converges towards a limit function $g_h ^{\infty}(x,v)$ when the $n$ goes to $\infty$. Moreover, this solution is close to the exact limit function $g^\infty(x,v)$ with an error estimate that scales in $h $ for the Lie splitting, and in $h^{2}$ for the Strang splitting. 
Note that our results also imply  convergence results in time   which are  uniform for positive times for 
$g^n(x,v)$ and give explicit convergence bounds for $f^n(x,v)$  in $\Hc^s$ (Sobolev space, see \eqref{defsob} below) that depend on the final time $T$ in a polynomial way. 

The main idea of our proof can be compared  with the classical {\em backward error analysis} methods widely used in Geometric Numerical Integration, see for instance \cite{HLW}, \cite{Reic04}: we express the numerical solution as the exact solution of a continuous Vlasov type  equation with time dependent kernel (with a poor regularity in time). Usually for Hamiltonian systems, the analysis has to be refined to make this equation independent of the time, implying the existence of a modified energy that is preserved by the numerical scheme. This ``time averaging" introduces in general a remaining error term which is exponentially small (with respect to the time step) for finite dimensional systems (see \cite{BG94}, \cite{HLW}, \cite{Reic04}, \cite{Reich99}) or requires the use of a CFL (Courant-Friedrichs-Lewy) condition for semilinear Hamiltonian equations  to be controlled (see \cite{FG11},\cite{F11}). 

Here the situation is completely different. The long time behavior of the solution is essentially controlled by a time convergent integral, which is a consequence of the dispersive effect of the free flow and ensures the existence of the continuous limit function $g^{\infty}(x,v)$ (The {\em Landau damping} effect, see \cite{Mouhot-Villani}, \cite{Bedrossian-Masmoudi-Mouhot}, \cite{RF}). As we will observe in the next section, the effect of the splitting approximation is essentially to discretize this convergent integral. As the integrand converges algebraically when the time goes to infinity, the numerical solution also yields a convergent time integral, even if the time appears in a discontinuous way in the evolution equation. 

The proof of the uniform convergence estimates is based on a similar argument, but requires slightly more regularity  for the functions than for the continuous case. 

\section{Landau damping for the Vlasov-HMF model, main result}

Before stating our main result, we first recall the scattering result derived in \cite{RF} (see also \cite{Mouhot-Villani}, \cite{Bedrossian-Masmoudi-Mouhot} for similar result with analytic or Gevrey regularity that are valid for much more singular interaction potentials).

We work in the following weighted Sobolev spaces: for a given
$\nu > 1/2,$ we set 
\begin{equation}
\label{defsob}
\Norm{f}{\Hc^s_{\nu}}^2 = \sum_{|p| + |q| \leqslant s} \int_{\T \times \R} (1 + |v|^2)^{\nu} | \partial_x^p \partial_v^q f|^2 \dd x \dd v , 
\end{equation}
and we shall denote by $\Hc_{\nu}^s$ the corresponding function space. We  shall  denote  by $\hat \cdot$ or $\mathcal{F}$ the Fourier transform  on $\mathbb{T} \times \mathbb{R}$  given by  
\begin{equation}
\label{fourier}
\hat f_k(\xi) = \frac{1}{2 \pi} \int_{\T \times \R} f(x,v) e^{-ikx - i \xi v} \dd x \dd v.
\end{equation}
 We shall need a stability property of the reference state $\eta$ in order to control the linear part of the Vlasov equation \eqref{eq:vp3}.
 Let us denote by $\eta = \eta(v)$ the  spatially homogeneous stationary state and let us define the functions 
\begin{equation}
\label{kernel}
K(n,t)= - np_{n}\,  nt \,   \hat{\eta}_{0} (nt), \quad  K_{1}(n,t)= - np_{n}\,  nt \,   \hat{\eta}_{0} (nt) \mathds{1}_{t \geq 0}, \quad t \in \mathbb{R}, \quad n \in \Z,
\end{equation}
where $(p_{k})_{k \in \Z}$ are the Fourier coefficients of the kernel $P(x)$. We shall denote by ${\hat{K}_{1}(n,\tau )}= {\int_{\mathbb{R}} e^{-i \tau  t} K_{1}(n,t) \, \dd t}$ the Fourier
transform of $K_{1}(n, \cdot)$. 
We shall assume that  $\eta$ satisfies the following condition 
\begin{equation}
% \label{eq:penrose}
\nonumber
{\bf (H)}  \quad    \eta(v) \in \Hc^5_{3} \quad  \mbox{and } \quad  \exists\,  \kappa>0, \quad   \inf_{ \mathrm{Im}\, \tau  \leq 0 }  | 1-  \hat{K}_{1}(n,\tau ) | \geq \kappa, \quad n= \pm1.
\end{equation}
%{\color{red} validit\'e de l'hypoth\`ese sous condition usuelle de monotonie etc \`a discuter, bonne hypoth\`ese dans le cas Gevrey ?}
Note that  thanks to the localization property of $\eta$ in the first part of the assumption, the Fourier transform of $K$ can be indeed continued in the half plane
  $\mathrm{Im}\, \tau  \leq 0$.
 Here, the assumption is particularly simple due to the fact that for our kernel, there are only two non-zero Fourier
 modes.
This assumption is very similar to the one used in \cite{Mouhot-Villani}, \cite{Bedrossian-Masmoudi-Mouhot} and can be related to the standard  statement of the Penrose criterion.
 In particular it is verified  for the states $\eta(v)= \rho(|v|)$ with $\rho$ non-increasing which  are also known to be Lyapounov stable for the nonlinear equation
 (see \cite{Marchioro-Pulvirenti}).

We also use the notation $\langle x \rangle = (1 + |x|^2)^{1/2}$ for $x \in \R$. 
In \cite{RF}, the following result is proved: 
\begin{theorem}
\label{maintheo}
Let us fix $s \geq 7$, $\nu >1/2$ and assume that $\eta\in \Hc^{s+4}_{\nu}$ satisfies the assumption ${\bf (H)}$. Assume that $g(0,x,v)$ is in $\Hc^s_{\nu}$.  Then there exists $\eps_{0}>0$ and a constant $C>0$ such that for every $\eps \in (0, \eps_{0}]$  there exists $g^\infty(x,v) \in \Hc^{s-4}_{\nu}$ such that for all $r \leq s - 4$ and $r \geq 1$, 
\begin{equation}
\label{scat}
\forall\, t \geq 0, \quad \Norm{g(t,x,v) - g^{\infty}(x,v)}{\Hc^{r}_{\nu}}  \leq \frac{C}{\langle t \rangle^{s - r-3}}. 
\end{equation}
\end{theorem}

In this paper, we prove the following semi-discrete version of the previous result:  
\begin{theorem}
\label{debussy}
Let us fix $s \geq 7$, $\nu>1/2$ and assume that $\eta\in \Hc^{s+4}_{\nu}$ satisfies the assumption ${\bf (H)}.$ Assume that $g(0,x,v)$ is in $\Hc^s_{\nu}.$ For a time step $h $,  let $g^n(x,v)$, $n \geq 0,$ be the sequence of functions defined by the formula \eqref{cimarosan} from iterations of the splitting methods \eqref{lie1} (Lie), or \eqref{strang} (Strang),  with $g^0(x,v) = g(0,x,v)$.   Then there exists $\eps_{0}>0$, $h _0 > 0$ and a constant $C > 0$ such that for every $\eps \in (0, \eps_{0}]$ and every $h  \in (0,h _0]$, there exists $g_h ^\infty(x,v) \in \Hc^{s-4}_{\nu}$ such that for all $r \leq s - 4$ and $r \geq 1$, 
\begin{equation}
\label{scatnum}
\forall\, n \geq 0, \quad \Norm{g^n(x,v) - g_h ^{\infty}(x,v)}{\Hc^{r}_{\nu}}  \leq \frac{C}{\langle nh \rangle^{s - r-3}}. 
\end{equation}
If moreover $\nu >3/2$  and $s \geq 8$,  we have for the Lie splitting methods \eqref{lie1} the estimate
\begin{equation}
\label{convergence1}
\Norm{g^n(x,v) - g(nh ,x,v)}{\Hc^{s-6}_{\nu-1}}  \leq C  h  \hspace{3mm} \forall n\in \mathbbm{N},
\end{equation}
where $g(t,x,v)$ is the solution \eqref{cimarosa} associated with the continuous equation with the same initial value.\\
In the case of the Strang splitting method \eqref{strang}, we have if $\nu >5/2$ and $s \geq 9$ the  second order estimate
\begin{equation}
\label{convergence}
\Norm{g^n(x,v) - g(nh ,x,v)}{\Hc^{s-7}_{\nu-2}}  \leq C  h^{2}  \hspace{3mm} \forall n\in \mathbbm{N}.
\end{equation}
\end{theorem}

Let us make the following comments: 

\medskip 
 {\bfÊa) }ÊThe estimates \eqref{convergence1} and \eqref{convergence} exhibit a convergence rates in time of order $1$ and $2$ respectively for the numerical solutions. These estimates hold  {\em uniformly in time}. Note however that these results do not imply convergence results uniform in time for the functions $f^n(x,v)$ to $f(nh ,x,v)$ given by the splitting methods \eqref{lie1} and \eqref{strang}. It is easy to check, using the formula $f^n(x,v) = \eta(v) + \varepsilon g^n(x - n h  v,v)$ that we have an estimate of the form 
\begin{equation}
\notag
\begin{split}
\forall\, n \geq 0, \quad \Norm{f^n(x,v) - f(nh ,x,v)}{\Hc^{s-6}_{\nu-1}}  \leq C \varepsilon \langle nh  \rangle^{s-6} h.  \\
\end{split}
\end{equation}
for the Lie splitting methods \eqref{lie1}. In the case of the Strang splitting \eqref{strang}, we have from the same arguments:
\begin{equation}
\notag
\begin{split}
\forall\, n \geq 0, \quad \Norm{f^n(x,v) - f(nh ,x,v)}{\Hc^{s-7}_{\nu-2}}  \leq C \varepsilon \langle nh  \rangle^{s-7} h ^2. \\
\end{split}
\end{equation}
Hence we obtain convergence results which are global in time only if we measure the error in $L^2$. If we measure the error in $H^\sigma$, $\sigma>0$,
then for a fixed time horizon $n h  \leq T$, the error grows like $T^\sigma$. This is however better than the rough $ e^T,$ estimate that is usually obtained through Gronwall type arguments (Note that convergence results can be found in \cite{Einkemmer1} for the case of compactly supported data, and in \cite{casas} for the Vlasov-Poisson case). 
\\
Let us also mention as an easy consequence of \eqref{scat}, \eqref{scatnum}, \eqref{convergence1} and \eqref{convergence}, that the following estimates hold for the limit state of the equation: For the Lie splitting
 \begin{equation}
 \notag
 \begin{split}
&\| g^{\infty}(x,v)-g^{\infty}_{h}(x,v) \|_{\Hc^{s-6}_{\nu-1}} \leq Ch, \\
\end{split}
\end{equation}
and for the Strang splitting
 \begin{equation}
 \notag
 \begin{split}
&\| g^{\infty}(x,v)-g^{\infty}_{h}(x,v) \|_{\Hc^{s-7}_{\nu-2}} \leq Ch^{2}.\\
\end{split}
\end{equation}

\medskip 
{\bf b)} The long time behavior of the exact solution \eqref{cimarosa} is essentially controlled by a time convergent integral (see \cite{RF}). We shall see (Proposition \ref{klanski} below) that the splitting method provides a discretization of this integral, but essentially without changing the decay in time of the integrand. Thus the numerical solution also yields a convergent time integral, even if the time appears in a discontinous way, giving us the long time behavior.\\
Moreover, this discretization is performed by rectangle methods in the case of Lie splittings, and by the midpoint rule in the case of Strang splitting. Estimates \eqref{convergence1} and \eqref{convergence} reflect the respective accuracies of these two methods. The second order estimate requires a more refined analysis than the first order, for it is obtained by tracking the cancellations provided by the midpoint rules. We mention that this result remains true for Strang splitting of the form \eqref{strang} where the role of $T$ and $P$ are exchanged but the complete proof is given for \eqref{strang} only (the time integration rule being the trapezoidal rule and the arguments identical for both cases).   \\
Finally, let us mention that the proofs of the convergence results (for Lie or Strang) widely use the long time behavior of both the exact and discrete solutions, in particular uniform bounds on their regularity. This can be understood as stability results for the numerical schemes. The convergence results are essentially the combination of these stability results, and the accuracy of the discretization of the integral.

\medskip 
{\bf c)} Our results hold only for time discretization of the equation. Fully discrete scheme including for example an interpolation procedure at each step (semi-Lagrangian methods) traditionally exhibit {\em recurrence} phenomena due the discretization in the $v$ variable. Indeed, the Landau damping effect  reflects essentially the fact that the solution of the free Vlasov equation is a superposition of travelling wave in the Fourier variable $\xi$. At the discrete level, the $\xi$ variable is only discretized by a finite number of points which causes numerical interactions of these travelling waves preventing the mixing effect to occur for very long time. Typically,  the previous result is hence valid {\em a priori} for a time of order $\mathcal{O}(1/\delta v)$ only, if $\delta v$ is the size of the mesh variable in $v$. Solutions exist to remedy these difficulties, for examples by putting absorbers in the Fourier spaces, see for instance \cite{Einkemmer2}. The analysis of these space discretization effects will be the subject of further studies.

\medskip 

Let us finally explain how the previous  scattering results imply Landau damping effects for the solution $f(t,x,v)$. Let us recall the following elementary Lemma: 

\begin{lemma}
\label{lememb}
For every $\alpha, \, \beta,\,  \gamma,  \,  s \in \mathbb{N}$ with  $\alpha + \beta = s$, and $\gamma< \nu - {1 \over 2}$. we have the following inequality: 
\begin{equation}
\label{emb1}
\forall\, k \in \Z, \quad \forall\, \xi \in \R, \quad 
|\partial_\xi^\gamma\hat{f}_k(\xi)| \leqslant 2^{s/2}C(\nu) \langle k\rangle^{- \alpha} \langle \xi\rangle^{- \beta} \Norm{f}{\Hc^{s}_{\nu}},
\end{equation}
where $C(\nu)$ depends only on $\nu>1/2$.  
\end{lemma}
\begin{proof}
We have by  using the  Cauchy-Schwarz inequality that 
\begin{eqnarray*}
\big| k^\alpha \xi^\beta  \partial_\xi^\gamma\hat{f}_k(\xi) \big| &=&  \frac{1}{2\pi} \left|\int_{\T \times \R} \partial_x^\alpha \partial_v^\beta(v^\gamma f(x,v)) e^{-ikx} e^{-iv\xi} \dd x  \dd v\right| \\
&\leqslant& C  \Norm{f}{\Hc^s_{\nu}} \Big( \int_\R (1 + |v|^2)^{\gamma -\nu} \dd v \Big)^{1/2}. 
\end{eqnarray*}
The previous inequality with $\alpha = \beta = 0$ yields the result when $k = 0$ or $|\xi | \leqslant 1$ and we conclude by using 
$\langle x \rangle \leqslant 2^{\alpha/2} |x|^2$ for $|x| > 1$ and the fact that $\nu - \gamma > 1/2$. 
\end{proof}

As a consequence of Theorem \ref{debussy}, we have the nonlinear Landau damping effect for the semi-discrete solution: The functions $g^n(x,v)$ being bounded in $\Hc^{s-4}_{\nu}$,  $f^n(x,v) = \eta(v) + \varepsilon r^n(x,v) = \eta(v) + \varepsilon g^n(x- n h  v,v)$ satisfy 
$$
\forall\, k \in \Z^*, \quad \forall \xi \in \R,\quad \forall\, \alpha + \beta  = s-4, \quad  |\hat f_k^n(\xi) | = \varepsilon|Ê\hat g^n_k(\xi + kn h   )  |\leq \frac{C \varepsilon}{\langle \xi + kn h   \rangle^{\alpha} \langle k \rangle^\beta}, 
$$
the last estimate being a consequence of  the  embedding  Lemma \ref{lememb}.
 This yields that  for every $k \neq 0$, $ \hat f_k^n(\xi)$ tends to zero when $nh \to \infty$ with a polynomial rate, but with a speed depending on $k$. 
 
Moreover,  by setting 
$$
\eta_h ^\infty(v) := \eta(v) + \frac{\varepsilon}{2\pi}\int_\T g_h ^{\infty}(x,v) \dd x, 
$$
we have by the previous  Lemma \ref{lememb} that  for $r \leq s - 4$, 
$$
\forall\, \xi \in \R, \quad 
|\hat f_0^n(\xi) - \hat \eta_h ^{\infty}(\xi) | \leq \frac{C}{\langle \xi \rangle^{r} \langle n h   \rangle^{s- r - 3}}. 
$$
In other words,  $f^n(x,v)$ converges weakly towards $\eta^\infty_h (v)$. Moreover, this weak limit $\eta^\infty_h (v)$ is $\mathcal{O}(h)$ for Lie splittings (or $\mathcal{O}(h^2)$ for Strang splitting) close to the exact limit 
$$
\eta^\infty(v) := \eta(v) + \frac{\varepsilon}{2\pi}\int_\T g^{\infty}(x,v) \dd x, 
$$
which exists by Theorem \ref{maintheo}.

\section{Backward error analysis} 
The unknown 
 $g(t,x,v )$ defined in \eqref{cimarosa} is solution of the equation
\begin{equation}
\label{eq:vp3}
\partial_t g = \{\phi(t,g(t)), \eta\}  + \varepsilon \{ \phi(t,g(t)),g \}. 
\end{equation}
where 
\begin{equation}
\label{eq:phi}
\phi(t,g) (x,v)= \int_{\R \times \mathbb{T}} ( \cos(x-y + t(v-u)) ) g(y,u) \dd u \dd y,
\end{equation}
 and  $\{f,g\} = \partial_x f \partial_v g - \partial_v f \partial_x g$ is the usual microcanonical Poisson bracket. 
 In Fourier space, we have the expression:
 $$
\phi(t,g) = \frac12  \sum_{k \in \{ \pm 1\}}e^{ikx}e^{iktv} \hat g_{k}(kt).
$$
In the evolution of the solution $g(t,x,v)$ of \eqref{eq:vp3}, 
an important role is played by the quantity
\begin{equation}
\label{eq:zetak}
\zeta_k(t) = \hat{g}_k(t,kt), \quad k  \in \{ \pm 1 \}, 
\end{equation}
such that 
$$
\phi(t,g(t)) = \frac12  \sum_{k \in \{ \pm 1\}}e^{ikx}e^{iktv} \zeta_k(t).
$$
Note that  for $k \neq 0$,   $ \zeta_{k}(t)$ is the Fourier coefficient in $x$ of the density $\rho(t,x)= \int_{\mathbb{R}} f(t,x,v) \, dv$.

The following result shows that semi-discrete solution $g^n(x,v)$ also satisfies an equation of the form \eqref{eq:vp3}, but with a discontinuous dependence with respect to the time: 

\begin{proposition}
\label{klanski}
For $h  > 0$, the solution $g^{n}(x,v)$ given by the splitting method  \eqref{lie1}  (Lie) or \eqref{strang} (Strang), and formula \eqref{cimarosan} coincides at times $t = nh $ with the solution $\gs(t,x,v)$ of the equation 
\begin{equation}
\label{julienclerc}
\partial_t \gs = \{Ê\Phi_h (t,\gs(t)), \eta\}  + \varepsilon \{ \Phi_h (t,\gs(t)),\gs(t) \}, 
\end{equation}
where $\Phi_h (t,\gs (t)) = \phi(\ssf_h (t),\gs(t))$ with the definition of $\phi$ given in \eqref{eq:phi}.\\
In the case of Lie splittings  \eqref{lie1} , $\ssf_{h}(t)$ are given by the formulas
\begin{equation}
\label{prince2}
\ssf_h (t)  = \left\lfloor \frac{t}{h }\right\rfloor h   + h,
\quad \mbox{and} \quad 
\ssf_h (t)  = \left\lfloor \frac{t}{h }\right\rfloor h, 
\end{equation}
respectively. 
In the case of the Strang splitting method \eqref{strang}, $\ssf_{h}(t)$ is given by the formula
\begin{equation}
\label{prince}
\ssf_h (t)  = \left\lfloor \frac{t}{h }\right\rfloor h   + \frac{h }{2}. 
\end{equation}

\end{proposition}

\begin{proof} We prove the result in the case of Strang splitting \eqref{strang}-\eqref{cimarosan}, the proof being analogous for Lie splittings.\\
By definition, the function $f^n(x,v)$ satisfies the recurrence relation  \eqref{strang}. 
Hence, we have (using the linearity of $\varphi_T^t$ and the fact that $\varphi_T^t(\eta) = \eta$ for all $t \in \R$) 
\begin{eqnarray*}
\eta + \varepsilon g^n &=& \varphi_T^{-n h } (f^n)\\
&=&\varphi_T^{-n h } \circ \varphi_T^{h /2} \circ \varphi_P^h  \circ \varphi_T^{h /2} (f^{n-1})\\
&=& \varphi_T^{-n h  + h /2} \circ \varphi_P^h  \circ \varphi_T^{h /2 + (n-1)h } \varphi_T^{- (n-1)h } (f^{n-1})\\
&=& \varphi_T^{-n h  + h /2} \circ \varphi_P^h  \circ \varphi_T^{n h  - h /2} (\eta + \varepsilon g^{n-1}). 
\end{eqnarray*}
Now we verify that for $t \in [0 ,h ]$, the application $t \mapsto \varphi_T^{-n h  + h /2} \circ \varphi_P^t \circ \varphi_T^{n h  - h /2}(\eta + \varepsilon g^{n-1})$ is the solution of the equation 
$$
\partial_t \tilde \gs =\{ \phi(\ssf_h (t),\tilde \gs),\tilde \gs \}.
$$
 with inital data $\tilde \gs(0) = \eta + \varepsilon g^{n-1}$. Using the fact that $\eta$ is a stationary state of the equation, we easily get the result. 
\end{proof}

For notational convenience, we will often  write in the following $\ssf(t)$ instead of $\ssf_h (t)$. As in \eqref{eq:zetak}, we define 
\begin{equation}
\label{eq:psik}
\zs_k(t) = \hat{\gs}_k(t,k \ssf(t)), \quad k  \in \{ \pm 1 \}, 
\end{equation}
such that 
$$
\Phi_h (t,\gs(t)) = \phi(\ssf(t),\gs(t)) = \frac12  \sum_{k \in \{ \pm 1\}}e^{ikx}e^{ik\ssf(t)v} \zs_k(t).
$$
Of course, we expect  the $\zs_k(t)$ to be approximations of the terms $\zeta_k(t)$ defined in \eqref{eq:zetak}. 

\begin{lemma}
\label{approx}
Let $h _0> 0$ be given. There exist two constants $c$ and $C> 0$ such that for all  $h \in (0,  h _0]$ and all $t > 0$, 
$$
c \langle t \rangle \leq \langle \ssf_h (t) \rangle \leq C \langle t \rangle 
$$
and for all $t$ and $\sigma$, 
$$
 c\langle t \pm \sigma \rangle \leq \langle \ssf_h (t) \pm \ssf_h (\sigma) \rangle \leq C \langle t \pm \sigma \rangle. 
$$

\end{lemma}
\begin{proof}
For $t \in \R$, we can write $t = n h  + \mu$ with $\mu \in [0,h )$. In the case of Strang splitting \eqref{prince}, we thus have 
$\ssf(t ) = n h  + h /2 = t + h/2 - \mu$. Hence we have $t - \ssf(t) \in [-h/2,h/2)$ which clearly implies the first inequality. The second is proved using the fact that with similar calculations $t \pm \sigma = \ssf(t) \pm \ssf(\sigma) + \mathcal{O}(h)$. The proof is analogous for Lie splittings \eqref{prince2}. \end{proof}

As in \cite{RF} we introduce the  weighted norms: 
 \begin{equation}
 \label{eq:NetM}  N_{T,s, \nu}(\gs)= \sup_{t \in [0,  T]}  {\| \gs(t) \|_{\Hc^s_{\nu}} \over \langle t \rangle^3}, \quad   M_{T, \gamma}(\zs)= \sup_{t \in [0,T]} \sup_{k \in \{\pm 1\}} \langle t \rangle^\gamma |\zs_{k}(t) |
 \end{equation}
and 
  \begin{equation}
  \label{Qdef}  Q_{T,s, \nu}(\gs)=  N_{T,s,\nu}(\gs) + M_{T, s-1}(\zs) + \sup_{[0, T ]} \|\gs(t) \|_{\Hc^{s-4}_{\nu}}.
  \end{equation}
 
We shall  prove the following result: 
\begin{theorem}
\label{maintheod}
Let us fix $s \geq 7$, $\nu >1/2$ and $R_{0}>0$ such that $Q_{0, s, \nu}(\gs) \leq R_{0},$ and assume that $\eta\in \Hc^{s+4}_{\nu}$ satisfies the assumption ${\bf (H)}$.  Then there exist $R>0$, $h_0 > 0$ and $\eps_{0}>0$ such that for every $\eps \in (0, \eps_{0}]$, $h \in (0,h_0]$ and 
 for every $T\geq 0$, the solution of \eqref{julienclerc} satisfies the estimate
 $$ Q_{T, s,\nu}(\gs) \leq R.
 $$
\end{theorem}
This result is a semi-discrete version of the main Theorem in \cite{RF} where  the same  norms are used  to control the solution $g(t)$ of the equation \eqref{eq:vp3}. Let us also mention that it holds for any of the three formulas \eqref{prince}-\eqref{prince2} defining $\ssf(t),$ the only property being used is the fact that $\ssf(t)$ satisfies Lemma \ref{approx}.
%
%  .
%

%We consider the equation under the form 
%\begin{equation}
%\label{eq:vp3}
%\partial_t g =  \{Ê\phi(t,g), \eta\}+ \varepsilon \{ \phi(t,g),g \}, 
%\end{equation}
%
%
%  
%%Note that this function can also be written 
%%$$
%%\frac12 \big( e^{i x  + i (M+t) v} g_1 (t,M + t) + e^{- i x - i (M+t)v } g_{-1} (t,-M - t))
%%$$
%Let us  set 
%$$
%\zeta_k(t) = \hat{g}_k(t,kt), 
%$$
%so that 
%$$
%\phi(t,g) = \frac12 \sum_{k \in \{ \pm 1\}}e^{ikx}e^{iktv} \zeta_k(t)
%$$

\section{Estimates}
   We fix now $s \geq 7$ and $R_0$ as in the previous Theorem. 
   In the following a priori estimates, $C$ stands for a number which may change from line to line and which is independent of $R_{0}$, $R$, $h$, $\eps$ and $T$.
   
%In Fourier space, the equation \eqref{eq:vp3} can be written after integration in time, 
%\begin{multline}
%\label{fourier1}
%\hat g_n(t,\xi) = \hat g_n(0,\xi)  +  \int_{0}^t p_n \zeta_n(\sigma) \hat \eta_0(\xi - n \sigma)( n^2 \sigma - n \xi)Ê\dd \sigma
% \\
%+ \varepsilon \sum_{k \in \Z } p_k \int_0^t \zeta_k(\sigma)  \hat g_{n-k}(\sigma,\xi - k\sigma) (  n k \sigma - k \xi)Ê\dd \sigma, 
%\end{multline}
%for all $(n,\xi) \in \Z \times \R$, 
%with $p_k = \frac12$ for $k \in \{\pm 1\}$ and $p_k = 0$ for $k \neq \pm 1$, and where the $\zeta_k(t)$ are defined by \eqref{eq:zetak}. 
%

\subsection{Estimate of $M_{T, s-1}(\zs)$ }
Towards the proof of Theorem \ref{maintheod}, we shall first estimate $\zs_{k}(t)$, $k=\pm1$.
 \begin{proposition}
 \label{propzeta}
 Assuming that $\eta \in \Hc^{s+2}_{\nu}$ verifies the assumption {\bf (H)}, then there exist $C>0$ and $h_0 > 0$ such that for every $T>0$ and $ h \in (0,h_0]$, every  solution of \eqref{julienclerc}  such that  
 $ Q_{T, s,\nu} (\gs) \leq R$ enjoys the estimate
 \begin{equation}
 \label{youyou} 
 M_{T, s-1}(\zs) \leq C \big(  R_{0} + \eps R^2 \big).
 \end{equation}
 \end{proposition}

\begin{proof}
 The main ingredient of the proof of the previous result is to write the equation \eqref{julienclerc} in Fourier: 
\begin{multline}
\label{fourier2}
\hat \gs_n(t,\xi) = \hat \gs_n(0,\xi)  +  \int_{0}^t p_n \zs_n(\sigma) \hat \eta_0(\xi - n \ssf(\sigma))( n^2 \ssf( \sigma) - n \xi)Ê\dd \sigma
 \\
+ \varepsilon \sum_{k \in \Z } p_k \int_0^t \zs_k(\sigma)  \hat \gs_{n-k}(\sigma,\xi - k \ssf(\sigma)) (  n k \ssf(\sigma) - k \xi)Ê\dd \sigma, 
\end{multline}
for all $(n,\xi) \in \Z \times \R$, 
with $p_k = \frac12$ for $k \in \{\pm 1\}$ and $p_k = 0$ for $k \neq \pm 1$, and where the $\zs_k(t)$ are defined by \eqref{eq:psik}.
Setting $\xi = n \ssf(t)$ in \eqref{fourier2}, the equation satisfied by $(\zs_n(t))_{n=\pm1}$ can be written under the almost closed form 
\begin{multline}
\label{zetanp1}
\zs_n(t) = \hat \gs_n(0,n \ssf(t)) -  \int_{0}^t p_n \zs_n(\sigma)  \hat \eta_0(n(\ssf(t) - \ssf(\sigma))) n^2 (\ssf(t) - \ssf(\sigma) )Ê\dd \sigma\\
- \varepsilon   \sum_{k \in \{ \pm 1\}} p_k \int_0^t \zs_k(\sigma)  \hat \gs_{n-k}(\sigma, n \ssf(t) - k \ssf(\sigma)) k n(\ssf(t) - \ssf(\sigma)) Ê\dd \sigma. 
\end{multline}
\begin{remark}
\label{constantpm}
Note that, for every $m\in \mathbbm{N},$  we have for  $t\in[mh,(m+1)h]$ the formula
\begin{multline}
\notag
\zs_n(t) - \hat \gs_n(mh,n \ssf(t)) =  - \int_{mh}^{t} p_n \zs_n(\sigma)  \hat \eta_0(n(\ssf(t) - \ssf(\sigma))) n^2 (\ssf(t) - \ssf(\sigma) )Ê\dd \sigma\\
- \varepsilon   \sum_{k \in \{ \pm 1\}} p_k \int_{mh}^{t} \zs_k(\sigma)  \hat \gs_{n-k}(\sigma, n \ssf(t) - k \ssf(\sigma)) k n(\ssf(t) - \ssf(\sigma)) Ê\dd \sigma.
\end{multline}~\\
As $\ssf(t)-\ssf(\sigma)=0$ for almost every $\sigma \in [mh,t],$ we notice that the function $t\mapsto \hat{\gs}_{n}(t,n\ssf(t)) = \zs_{n}(t)$ is constant on the small intervalls $[mh,(m+1)h].$ This is due to the fact that the electric field is constant during the evolution of \eqref{boston}. 
\end{remark}

To study the equation \eqref{zetanp1}, we shall first consider the corresponding  linear equation, that is to say that we  shall first see
 \begin{equation}
 \label{Fdef} F_{n}(t):=   \hat \gs_n(0,n\ssf(t))- \varepsilon   \sum_{k \in \{ \pm 1\}} p_k \int_0^t \zs_k(\sigma)  \hat g_{n-k}(\sigma, n\ssf(t) - k\ssf(\sigma)) k n(\ssf(t)-\ssf(\sigma)) Ê\dd \sigma
 \end{equation}
  as a given source term and we shall study the linear integral equation
 \begin{equation}
 \label{volterrad}
  \zs_{n} (t) =  \int_{0}^t K(n, \ssf(t)-\ssf(\sigma)) \zs_{n}(\sigma) \,\dd\sigma + F_{n}(t) \quad n= \pm 1, 
 \end{equation}
 where the kernel $K(n,t)$ has been introduced in \eqref{kernel}.
 We rewrite this equation as 
  \begin{equation}
 \label{beethov}
  \zs_{n} (t) =  \int_{0}^t K(n, t - \sigma) \zs_{n}(\sigma) \,\dd\sigma + F_{n}(t) + G_n(t)\quad n= \pm 1,
 \end{equation}
 where 
 \begin{equation}
 G_n(t) = \int_{0}^t \Big( K(n, \ssf(t)-\ssf(\sigma)) - K(n, t - \sigma)\Big)    \zs_{n}(\sigma)\,\dd\sigma 
 \end{equation}
 is an error term due to the time discretization. 
To study the linear equation \eqref{volterrad}-\eqref{beethov}, we use the result given by the following Lemma. The proof of this Lemma can be found in \cite{RF}. For completeness, we recall it in Appendix of the present paper. 
  \begin{lemma}
  \label{lemvolterra}
 Let $\gamma \geq 0$, and  assume that $\eta \in \Hc^{\gamma+3}_{\nu}$ satisfies ${\bf (H)}$. Then,
    there exists  $C>0$ such   for every $T \geq 0$, we have
   $$ M_{T, \gamma}( \zs) \leq C M_{T, \gamma}( F + G).$$
     \end{lemma}

%Let us postpone the proof of the Lemma and finish the proof of Proposition \ref{propzeta}.

 From this Lemma and \eqref{emb1}, we first  get that
 \begin{equation}
 \label{est1:1} M_{T, s-1}(\zs) \leq  C\big( \|\gs(0) \|_{\Hc^{s}_{\nu}} + M_{T,s-1}(G) + \eps  M_{T, s-1}(F^1) + \eps  M_{T, s-1}(F^2) \big) 
 \end{equation}
 where
 \begin{align*}
 &  F^1_{n}(t) =  n^2 p_{-n} \int_{0}^t \zs_{-n}(\sigma) \hat \gs_{2n}( \sigma, n(\ssf(t) + \ssf(\sigma))) (\ssf(t)- \ssf(\sigma)) \, \dd\sigma, \quad n= \pm1, \\
 &  F^2_{n}(t) = -n^2 p_{n} \int_{0}^t \zs_{n}(\sigma) \hat \gs_{0}( \sigma, n(\ssf(t)- \ssf(\sigma))) (\ssf(t)- \ssf(\sigma)) \,\dd\sigma, \quad n= \pm 1.
 \end{align*}
 Let us estimate $F^1_{n}$.  By using  again \eqref{emb1} and the definition \eqref{eq:NetM} of $N_{\sigma,s,\nu}$,  we get using Proposition \ref{klanski} that
 $$ |F^1_{n}(t)| \leq  C \int_{0}^t {  (\ssf(t)- \ssf(\sigma)) \langle \sigma \rangle^3 M_{\sigma, s-1}(\zs) N_{\sigma, s,\nu}(\gs)  \over \langle \sigma \rangle^{s-1}   \langle \ssf(t)+  \ssf(\sigma) \rangle^s} \, \dd \sigma \leq C {R^2 \over\langle  t \rangle^{s-1}} \int_{0}^{+ \infty} { 1 \over \langle \sigma \rangle^{s-4}}\, \dd \sigma \leq   C {R^2 \over\langle  t \rangle^{s-1}}$$
  provided  $s \geq  6$. This yields that for all $T \geq 0$
  $$ M_{T, s-1}(F^1) \leq CR^2.$$
  
  To estimate $F_{n}^2$, we split the integral into two parts: we write
   $$  F^2_{n}(t) =  I^1_{n}(t) + I_{n}^2 (t)$$ 
   with
    \begin{align*}
 &    I^1_{n}(t) = -n^2 p_{n} \int_{0}^{t \over 2} \zs_{n}(\sigma) \hat \gs_{0}( \sigma, n(\ssf(t)- \ssf(\sigma))) (\ssf(t)- \ssf(\sigma)) \, \dd\sigma, \quad n= \pm 1, \\
  &    I^2_{n}(t) = -n^2 p_{n} \int_{t \over 2}^{t } \zs_{n}(\sigma) \hat \gs_{0}( \sigma, n(\ssf(t)- \ssf(\sigma))) (\ssf(t)- \ssf(\sigma)) \, \dd\sigma, \quad n= \pm 1.
     \end{align*}
      For $I_{n}^1$, we proceed as previously, 
      $$ |I_{n}^1 (t) | \leq C R^2  \int_{0}^{t \over 2 } {  \langle \sigma \rangle^3( \ssf(t)- \ssf(\sigma))  \over \langle \sigma \rangle^{s-1}  \langle \ssf(t)-\ssf(\sigma )\rangle^{s}} \, \dd \sigma \leq  { C R^2 \over \langle t \rangle^{s-1}} \int_{0}^{+ \infty} { 1 \over \langle
       \sigma \rangle^{s- 4}} \, \dd \sigma$$
        and hence since $s \geq 6$, we have
   $$ M_{T, s-1}(I^1) \leq C R^2.$$
     To estimate $I_{n}^2$, we shall rather use the last factor in the definition of $Q_{T,s,\nu}$ in \eqref{Qdef}. By using again \eqref{emb1},  we write
     $$ | I_{n}^2(t) | \leq C\int_{t \over 2}^{ t  }  { M_{\sigma, s-1}(\zs) \over \langle \sigma \rangle^{s-1} }   {\| \gs(\sigma) \|_{\Hc^{s-4}_{\nu}} \over  \langle  \ssf(t)- \ssf(\sigma) \rangle^{s-5}} \, \dd \sigma
      \leq  {C R^2 \over \langle t \rangle^{s-1}  }\int_{0}^{+ \infty} { 1 \over \langle \sigma  \rangle^{s-5}}\, \dd \sigma \leq { CR^2 \over \langle t \rangle^{s-1}}$$
      and hence since $s \geq 7,$ we find again 
     $$ M_{T, s-1}(I^2) \leq { C R^2}.$$ 
     
Finally, we have 
\begin{eqnarray*}
\langle t \rangle^{s-1} |G_n(t)|   &\leq& M_{T,s-1}(\zs) \int_{0}^t \frac{ \langle t \rangle^{s-1}}{\langle \sigma\rangle^{s-1}} \left |  \int_{t - \sigma}^{\ssf(t)-\ssf(\sigma)}
| \partial_t K(n, \theta) |\dd \theta \right | \dd\sigma 
\\
&\leq & C
M_{T,s-1}(\zs) \int_{0}^t  \left| \int_{t - \sigma}^{\ssf(t)-\ssf(\sigma)}
\frac{ \langle t \rangle^{s-1}}{\langle \theta \rangle^{s+2}\langle \sigma\rangle^{s-1}}  \dd \theta  \right|\dd\sigma, 
\end{eqnarray*}
where we used the fact that  $\eta \in \Hc^{s+2}_{\nu}$  and Lemma \ref{lememb}. 
Now, since 
$$
\left| \int_{t - \sigma}^{\ssf(t)-\ssf(\sigma)}\frac{ \dd \theta}{\langle \theta \rangle^{s+2}}  \right|\leq  \frac{C h}{\langle t- \sigma  \rangle^{s+2}},
$$
 we get 
$$
M_{T,s-1} (G) \leq C h M_{T,s-1} (\zs)  \int_{0}^t \frac{\langle t \rangle^{s-1}}{\langle t - \sigma \rangle^{s+2}\langle \sigma \rangle^{s-1}} \dd \sigma. 
$$
As the integral 
$$
\int_{0}^t \frac{\langle t \rangle^{s-1}}{\langle t - \sigma \rangle^{s+2}\langle \sigma \rangle^{s-1}} \dd \sigma
\leq C \int_{0}^{t/2} \frac{1}{\langle t \rangle^{3}\langle \sigma \rangle^{s-1}} \dd \sigma + C \int_{t/2}^t \frac{1}{\langle t - \sigma \rangle^{s+2}} \dd \sigma.
$$
is uniformly bounded in time, we conclude that 
\begin{equation}
\label{G+1}M_{T,s-1} (G) \leq C h M_{T,s-1}(\zs).
\end{equation}
      By combining the last estimates and \eqref{est1:1}, we thus obtain \eqref{youyou}.
      $$   M_{T, s-1}(\zs) \leq  C\big(  R_{0} + h M_{T,s-1}(\zs) + \eps  R^2 \big).$$
      By taking $h \leq h_0$ small enough, 
       this ends the proof of Proposition \ref{propzeta}.  
\end{proof}

%It remains to prove Lemma \ref{lemvolterra}.
%\begin{proof}[Proof of Lemma \ref{lemvolterra}]
%
%We recall the following result: 
% \begin{lemma}
%  \label{lemvolterra}
% Let $\gamma \geq 0$, and  assume that $\eta \in \Hc^{\gamma+3}$ satisfies ${\bf (H)}$. Then,
%    there exists $C>0$ such   for every $T \geq 0$, we have
%   $$ M_{T, \gamma}( \zeta) \leq C M_{T, \gamma}( F).$$
%     \end{lemma}
%     
     
%{\color{red} il y a une petite variation par rapport \`a Villani-Mouhot pour ne pas passer par l'estimation de la norme $L^2$ de zeta par rapport \`a la norme $L^2$ du terme source
%qui fait perdre une puissance de $T$ suppl\'ementaire quand on repasse aux normes sup si on ne perd pas sur la rayon d'analycit\'e}

\subsection{Estimate of $N_{T, s,\nu}(\gs)$}
\begin{proposition}
 \label{propgs}
 Assuming that $\eta \in \Hc^{s+2}_{\nu}$ verifies the assumption {\bf (H)}, then there exists $C>0$ and $h_0 > 0$ such that for every $T>0$ and $h \in (0,h_0]$, every  solution of \eqref{julienclerc}  such that $Q_{T, s,\nu} (\gs) \leq R$ enjoys the estimate
   $$  N_{T, s,\nu}(\gs) \leq  C (R_{0}+  \eps R^2)  ( 1 + \eps R) e^{ C \eps  R}.$$
 \end{proposition}
 
 \begin{proof}
  To prove Proposition \ref{propgs}, we shall use energy estimates.
  We set $\mathcal{L}_t[\gs]$ the operator
 $$\Lc_t[\gs] f = \{\phi(t,\gs),f\}$$ such 
 that $\gs$ solves the equation 
 $$\partial_t \gs = \Lc_{\ssf(t)}[\gs(t)]( \eta + \varepsilon \gs).$$
For any linear operator $D$, we thus have by standard manipulations 
 that 
\begin{eqnarray*}
\frac{\dd}{\dd t} \Norm{D\gs(t)}{L^2}^2 &=& 2 \varepsilon  \langle D \gs(t), D( \Lc_{\ssf(t)} [\gs(t)] \gs(t)) \rangle_{L^2} + 2 \langle D\gs(t), D( \Lc_{\ssf(t)} [\gs(t)](\eta) )\rangle_{L^2}\\
&=& 2\varepsilon \langle D\gs(t),  \Lc_{\ssf(t)} [\gs(t)] D \gs(t) \rangle_{L^2}  + 2\varepsilon \langle D\gs(t), [D, \Lc_{\ssf(t)}[\gs(t)]] \gs(t)\rangle_{L^2}\\
&& + 2 \langle D\gs(t), D( \Lc_{\ssf(t)} [\gs(t)](\eta) )\rangle_{L^2},
\end{eqnarray*}
where $[D,\Lc_{\ssf(t)}]$ denotes the commutator between the two operators $D$ and $\Lc_{\ssf(t)}$.  
The first term in the previous equality vanishes since $\mathcal{L}_{\ssf(t)}[\gs]$ is the transport operator associated with a divergence free Hamiltonian vector field.  
Consequently,  we get that  
\begin{multline}
\label{energie1}
 {\dd \over \dd t}\Norm{D\gs(t)}{L^2}^2 \leqslant    2\varepsilon\ \Norm{D\gs(t)}{L^2} \Norm{[D,\Lc_{\ssf(t)}[\gs(t)]] \gs(t)}{L^2} 
\\+ 2 \Norm{ D\gs(t)}{L^2}\Norm{ D ( \Lc_{\ssf(t)} [\gs(\sigma)] (\eta) ) }{L^2}. 
\end{multline}
 To get the estimates of Proposition \ref{propgs}, we shall use the previous estimates with the operator 
$D = D^{m,p,q}$ defined as the   Fourier  multiplier by  $ k^p \xi^q  \partial^m_\xi$ for  $(m,p,q) \in \mathbb{N}^{3d}$ such that $p + q \leqslant s $, $m \leqslant \nu$ and the definition \eqref{defsob} of the $\Hc^{s}_{\nu}$ norm.  
  To evaluate the right hand-side of \eqref{energie1}, we shall use the following Lemma, whose proof is given in Appendix (see also \cite{RF}). 
  \begin{lemma}
  \label{lemcom}
   For   $p + q \leqslant  \gamma$ and $r \leqslant \nu$, and functions $h(x,v)$ and $g(x,v)$, we have the estimates
   \begin{eqnarray}
   \label{com1}
    & &  \| \big[ D^{r,p,q} ,  \mathcal{L}_{\sigma}[g] \big] h \|_{L^2} \leq C \big( m_{\sigma, \gamma + 1}(\zeta)  \|h\|_{\Hc^1_{\nu}} +  m_{\sigma, 2}(\zeta) \|h\|_{\Hc^\gamma_{\nu}}\big),  \\
  \label{com2}   & & \|D^{r,p,q} \big(   \mathcal{L}_{\sigma}[g] \big)h   \|_{L^2} \leq C \big(  m_{\sigma, \gamma + 1}(\zeta)  \|h\|_{\Hc^{1}_{\nu}}
  + m_{\sigma, 2}(\zeta) \|h\|_{\Hc^{\gamma + 1 }_{\nu}}\big),
      \end{eqnarray}
      for all $\sigma$, 
   where the sequence $\zeta_{k}$ is defined by   $\zeta_{k}= \hat g_{k}( k\sigma)$, $k \in \{\pm 1\}$, and where
   $$  m_{\sigma, \gamma}(\zeta) = \langle \sigma \rangle^{\gamma} \Big(\sup_{k \in \{\pm 1\}}  |\zeta_k|\Big),$$
   with a constant $C$ depending only on $\gamma$, and in particular, not depending on  $\sigma$.
   \end{lemma} 
  Let us finish the proof of  Proposition \ref{propgs}.  By using the previous lemma with $\gamma = s$, $\sigma= \ssf(t)$, $g= \gs(t)$ and $h= \gs(t)$ or 
    $h= \eta$, we obtain from   \eqref{energie1}    that
  \begin{multline*}  { \dd \over \dd t } \|\gs(t) \|_{\Hc^s_{\nu}}^2  \leq  C \langle t \rangle^2 m_{t, s-1}(\zs(t)) \big( \| \eta \|_{\Hc^1_{\nu}} + \eps  \|\gs(t)\|_{\Hc^1_{\nu}} \big)  \|\gs(t)\|_{\Hc^s_{\nu}}  \\+   { C \over \langle t \rangle^{s - 3}} m_{t, s-1}(\zs(t))
   \| \eta\|_{\Hc^{s+1}_{\nu}} \| \gs(t) \|_{\Hc^s_{\nu}} 
   +  { C \eps  \over \langle t \rangle^{s - 3}} m_{t, s-1}(\zs(t)) \|\gs(t) \|_{\Hc^s_{\nu}}^2.
   \end{multline*}
    This yields, using the fact that $M_{t,\gamma}(\zs) = \sup_{\sigma \in [0,t]} m_{\sigma,\gamma}(\zs(\sigma))$,
  $$ \|\gs(t)\|_{\Hc^s_{\nu}} \leq \|\gs(0)\|_{\Hc^s_{\nu}} + C \langle t \rangle^3 M_{t, s-1}(\zs) \big(  \|\eta\|_{\Hc^{s+1}_{\nu}}+ \eps R \big) +  C \eps R \int_{0}^t { 1 \over \langle \sigma \rangle^{s - 3}}
   \| \gs(\sigma) \|_{\Hc^{s}_{\nu}} \, \dd\sigma$$
    for $t \in [0, T]$.
     From the Gronwall inequality, we thus obtain
     $$ \|\gs(t)\|_{\Hc^s_{\nu}} \leq  \Big(\|\gs(0)\|_{\Hc^s_{\nu}} + C \langle t \rangle^3 M_{t, s-1}(\zs) \big(  \|\eta\|_{\Hc^{s+1}_{\nu}}+ \eps R \big)\Big) e^{ C \eps R \int_{0}^{ + \infty}{ \dd \sigma \over \langle \sigma \rangle^{s-3}}}.$$
      By using Proposition \ref{propzeta}, this yields 
      $$  N_{T, s,\nu}(\gs) \leq \Big( R_{0} +  C(R_{0}+  \eps R^2)  ( 1 + \eps R)  \Big) e^{ C\eps  R}.$$
       This ends the proof of Proposition \ref{propgs}.  
 \end{proof}

   \subsection{Estimate of $\| \gs \|_{ \Hc^{s-4}_{\nu}}$}
   To close the argument, it only remains to estimate  $\| \gs \|_{ \Hc^{s-4}_{\nu}}$.

    \begin{proposition}
    \label{propgs-4}
    Assuming that $\eta \in \Hc^{s+2}_{\nu}$ verifies the assumption {\bf (H)}, then there exists $C>0$ and $h_0>0$ such that for every $T>0$ and $h \in (0,h_0]$, every  solution of \eqref{julienclerc}  such that
 $ Q_{T, s,\nu} (\gs) \leq R$ enjoys the estimate
    $$ \| \gs(t) \|_{\Hc^{s-4}_{\nu}}  \leq  C\big( R_{0} +  \eps R^2) e^{ C \eps R}, \quad \forall t \in [0, T].$$
    \end{proposition}

 \begin{proof}
 We use again \eqref{energie1} with $D= D^{m,p,q}$ but now with $p+q  \leq s-4$.
  By using Lemma \ref{lemcom}, we find
  \begin{equation}
  \label{gs-4scat}
   {\dd \over \dd t} \|\gs(t)\|_{\Hc^{s-4}_{\nu}}^2 \leq  C m_{t, s-3}(\zs(t)) \big(  \| \eta\|_{\Hc^{s-3}_{\nu}} \|\gs(t)\|_{\Hc^{s-4}_{\nu}} +\eps \|\gs(t) \|_{\Hc^{s- 4}_{\nu}}^2  \big).\end{equation}
   This yields
   \begin{multline*}
    \|\gs(t)\|_{\Hc^{s-4}_{\nu}} \leq \|\gs(0)\|_{\Hc^{s-4}_{\nu}}  \\
    +  C \| \eta \|_{\Hc^{s-3}_{\nu}} M_{t, s-1}(\zs) \int_{0}^t { 1 \over \langle \sigma \rangle^2} \, \dd\sigma + C \eps M_{t, s-1}(\zs)
    \int_{0}^t { 1 \over \langle \sigma \rangle^2}  \|\gs(\sigma) \|_{\Hc^{s-4}_{\nu}} \, \dd \sigma.
    \end{multline*}
    By using Proposition \ref{propzeta}, we thus get
    $$ \|\gs(t)\|_{\Hc^{s-4}_{\nu}}    \leq C\big( R_{0} +  \eps R^2) + C \eps R  \int_{0}^t  { 1 \over \langle \sigma \rangle^2}\|\gs(\sigma) \|_{\Hc^{s-4}_{\nu}} \, \dd \sigma.$$
    From the Gronwall inequality, we  finally find
    $$ \| \gs(t) \|_{\Hc^{s-4}_{\nu}}  \leq C\big( R_{0} +  \eps R^2) e^{ C \eps R}.$$
     This ends the proof of Proposition \ref{propgs-4}.
 
 \end{proof}
 
 \subsection{Proof of Theorem \ref{maintheod} and estimate \eqref{scatnum}}
 The proof of Theorem \ref{maintheod} follows from the a priori estimates in Propositions \ref{propzeta}, \ref{propgs} and \ref{propgs-4} and a continuation argument.
  Indeed, by combining the estimates of these three propositions, we get that
  $$ Q_{T, s,\nu}(\gs) \leq  C (R_{0}+  \eps R^2)  ( 1  + \eps R)  e^{ C \eps  R}$$ 
  assuming that $Q_{T, s,\nu}(\gs) \leq R$. Consequently, let us choose $R$ such that
    $R> C R_{0}$, then for $\eps$ sufficiently small we have $ R >C (R_{0}+  \eps R^2)  ( 1  + \eps R)   e^{ C \eps  R}$ and hence by usual continuation argument, we obtain that
    the estimate $Q_{T, s,\nu}(\gs) \leq R$ is valid for all times.
    
To prove \eqref{scatnum}, let us define $g_h^\infty(x,v)$ by
     $$
     g^\infty_h(x,v) = \gs(0,x,v) + \int_0^{+\infty} \mathcal{L}_{\ssf(\sigma)}[\gs(\sigma)](\eta + \epsilon \gs(\sigma)) \dd \sigma, 
     $$
        To justify the convergence of the integral, we note that thanks to  \eqref{com2}, we have for all $\sigma$
     $$
     \left\|\mathcal{L}_{\ssf(\sigma)}[\gs(\sigma)](\eta + \epsilon \gs(\sigma))\right\|_{\Hc^{s-4}_{\nu}} \leq   C\Big( m_{\sigma,s-3}(\zs(\sigma)) \left\| \eta + \epsilon \gs(\sigma) \right\|_{\Hc^{1}_{\nu}} + m_{\sigma,2}(\zs(\sigma)) \left\| \eta + \epsilon\gs(\sigma) \right\|_{\Hc^{s-3}_{\nu}} \Big),
     $$
     where $\zs_{k}(t)=\hat{\gs}_{k}(t,k\ssf(t)),$ $k=\pm1,$ and
     $$
     m_{\sigma,\gamma}(\zs(\sigma)) = \langle \sigma \rangle^{\gamma}\left(\sup_{k=\pm 1}\left|\zs_{k}(\sigma)\right|\right).
     $$
     By interpolation, we have
     $$
     \|\eta + \epsilon\gs(\sigma)\|_{\Hc^{s-3}_{\nu}} \leq C \|\eta + \epsilon \gs(\sigma) \|_{\Hc^{s-4}_{\nu}}^{3 \over 4} \| \eta +\epsilon\gs(\sigma) \|_{\Hc^s_{\nu}}^{1 \over 4},
     $$
     and thus, using the bound provided by Theorem \ref{maintheod}, we have
      $$
      \|\eta + \epsilon \gs(\sigma)\|_{\Hc^{s-3}_{\nu}} \leq  C(R)   \langle \sigma \rangle^{ 3 \over 4}.
      $$
%      Also,
%      $$
%       \|\eta + \epsilon \gs(\sigma)\|_{\Hc^{1}} \leq  C(R).
%       $$
     Using again    Theorem \ref{maintheod}, we thus find that 
%       $$
%       \langle \sigma \rangle^{s-1}\left(\sup_{k=\pm 1}\left|\zs_{k}(\sigma)\right|\right) \leq C(R).
%       $$
%       Therefore
       \begin{equation}
       \label{enplus1}
        \left\|\mathcal{L}_{\ssf(\sigma)}[\gs(\sigma)](\eta + \epsilon \gs(\sigma))\right\|_{\Hc^{s-4}_{\nu}} \leq C(R) \left({ 1 \over \langle \sigma \rangle^{  2  } } +  { 1 \over   \langle \sigma \rangle^{s -3 - { 3 \over 4}} }\right),
       \end{equation}
        and $g_{h}^{\infty}(x,v)$ is then well defined, and belongs to $\Hc^{s-4}_{\nu}.$\\
        Since we have for all $t$
        $$
        \gs(t)-g_{h}^{\infty}(x,v) = \int_{t}^{+\infty} \mathcal{L}_{\ssf(\sigma)}[\gs(\sigma)](\eta + \epsilon \gs(\sigma)) \dd \sigma,
        $$
       we find by using again \eqref{enplus1} that 
      $$ \| \gs(t) - g_h^{\infty}\|_{\Hc^{s-4}} \leq C(R) \Big( \int_{t}^{+ \infty} { 1 \over \langle \sigma \rangle^{  2  } } +  { 1 \over   \langle \sigma \rangle^{s -3 - { 3 \over 4}} } \, \dd\sigma \Big)
       \leq {C(R)  \over  \langle  t  \rangle }.
       $$
      In a similar way, by using again \eqref{com2}, we have for $r \leq s-4$ and $r \geq 1$, 
   $$   \| \gs(t) - g_h^{\infty}\|_{\Hc^{r}_{\nu}} \leq C(R) \Big( \int_{t}^{+ \infty} { 1 \over \langle \sigma \rangle^{  s- r - 2  } } +  { 1 \over   \langle \sigma \rangle^{s -3}  }\, \dd\sigma \Big)
       \leq C(R) \big( { 1 \over  \langle  t  \rangle^{ s- r - 3 }  } +  { 1 \over   \langle  t  \rangle^{s - 4} } \big) \leq  {C(R)  \over  \langle  t  \rangle^{ s- r - 3 } }, $$ 
       which gives the result using the fact that $\gs(nh) = g^{n}$ the solution given by the numerical scheme.

\section{Proof of the convergence estimate \eqref{convergence1}}       
      We shall now prove the convergence estimate \eqref{convergence1}. Note that in view of the previous result and of the analysis in \cite{RF}, the functions $\gs(t,x,v)$ and $g(t,x,v)$ satisfy the same estimates. In particular, we can assume that $Q_{T,s, \nu}(\gs),$ $Q_{T,s, \nu-1}(\gs),$ $Q_{T,s, \nu}(g)$ and $Q_{T,s, \nu-1}(g)$ are both uniformly bounded by the same constant $R,$ provided that $\nu >3/2.$\\
      
      By using the equations \eqref{eq:vp3} and \eqref{julienclerc}, we get that $\delta = g- \gs$ solves the equation
      \begin{equation}
      \label{eqdiff}
      \partial_{t} \delta(t) = \{ \phi (t, \delta(t)), \eta \} +  \eps \{\phi(t, \delta(t)), g(t) \}  + \eps \{\phi(t, g(t)), \delta(t)\} -  \eps \{ \phi(t, \delta(t)), \delta(t) \} + \mathcal{R}
      \end{equation}
      with 
    \begin{equation}
    \label{Rhdef} 
      \mathcal{R}(t,x,v)=  \{ \phi(t, \gs (t)) -\phi(\ssf(t), \gs (t)), \eta \} +  \eps \{ \phi(t, \gs(t))- \phi(\ssf(t), \gs(t)), \gs(t) \}
     \end{equation}       
        and with zero initial data.
      It will be useful to use the expression of $\mathcal{R}$ in Fourier which reads
    \begin{multline}
    \label{Rfourier}
   \hat{\mathcal{R}}_{n}(t, \xi)
   =  np_{n}  \Big( \hat\gs_{n}(t, n t) \hat \eta_{0}(\xi- n t) (nt- \xi) - \hat  \gs_{n}(t, n \ssf(t)) \hat \eta_{0}(\xi- n \ssf(t)) (n\ssf(t) - \xi) \Big) \\
     +  \eps \sum_{k=\pm1}  kp_{k} \Big(  \hat\gs_{k}(t, k t) \hat \gs_{n-k} (t, \xi- kt) (nt- \xi)
      -    \hat\gs_{k}(t, k \ssf(t)) \hat \gs_{n-k} (t, \xi- k\ssf(t)) (n \ssf(t)- \xi) \Big).
    \end{multline}

Let $T$ be a positive real number.  By using the  weighted norms defined in  \eqref{eq:NetM}, we  now consider  for $s \geq 8$ and $\nu>3/2,$ the quantity
$$Q_{T,s-2, \nu-1}(\delta)= M_{T,s-3}(d) + N_{T,s-2, \nu-1}(\delta) + \sup_{t\in[0,T]} \left\|\delta(t) \right\|_{\Hc^{s-6}_{\nu-1}} $$
with $N_{T,s,\gamma}$ and $M_{T,\gamma}$ defined in \eqref{eq:NetM}, and where we set
$$ d_{k}(t)= \hat g_{k}(t,kt)- \hat \gs_{k}(t, kt).$$
%Recall that
%$$M_{T,s-2}(\zs-\zeta) = \sup_{t\in[0,T]}\sup_{k=\pm1} \langle t \rangle^{s-2} \left| \zs_{k}(t)-\zeta_{k}(t)\right|,$$
%and
%$$N_{T,s-1}(\gs-g) = \sup_{t\in[0,T]} \frac{\left\| \gs(t)-g(t)\right\|_{\Hc^{s-1}}}{\langle t \rangle ^{3}}.$$
We shall prove the following result:
\begin{proposition}
\label{ordre1}
For  $s\geq 8$, $\nu>3/2$,  assume that $\eta \in \Hc^{s+4}_{\nu}$ satisfies the assumption {\bf(H)}. Then there exists $R_{1}>0,$ $h_{0}>0$ and $\epsilon_{0}>0$ such that for every $h\in(0,h_{0}],$ every $\eps \in (0,\eps_{0}]$ and every $T\geq0,$ the solution of \eqref{eqdiff} satisfies the estimate
$$Q_{T,s-2,\nu-1}(\delta)\leq R_{1}h.$$
\end{proposition}
 Proposition \ref{ordre1} clearly implies the convergence estimate \eqref{convergence1}. It actually proves that the interpolation $\gs(t)$ of the sequence of functions $g^{n}(x,v)$ given by the splitting methods (Strang or Lie) is always at least an approximation of order one of the exact solution $g(t)$ at all times.  The proof will use the same steps 
 as in the proof of Theorem \ref{maintheod}.\\
 
 \begin{remark}
\notag
 We mentioned that $\zs_{k}(t)$ is expected to be an approximation of $\zeta_{k}(t).$ By Taylor expanding in $\xi,$ it is indeed clear that Proposition \ref{ordre1}, together with the uniform bound on $Q_{T,s,\nu}(\gs)$ and $Q_{T,s,\nu}(g),$ implies that
 $$\sup_{\substack{t\in \mathbbm{R} \\ k=\pm1}} |\zeta_{k}(t)-\zs_{k}(t)| = \mathcal{O}(h).$$
 In fact, we could have proved without any major difference Proposition \ref{ordre1} using the quantity $\tilde{d}_{k}(t)=\zeta_{k}(t)-\zs_{k}(t)$ instead of $d_{k}(t).$ The use of $d_{k}(t)$ will however be crucial to prove the second order estimate, since it is clear that the quantity $\tilde{d}_{k}(0)$ does not scale in $h^{2}.$
 \end{remark}
 
%An easy and useful consequence of proposition \ref{ordre1} is the following:
%\begin{corollary}
%\label{ordre1bis}
%Under the assumptions of proposition \ref{ordre1}, there exists $h_{0}>0$ and $\epsilon_{0}>0$ such that for every $h\in(0,h_{0}],$ $\epsilon \in (0,\epsilon_{0}]$ and for every $T\geq0,$
%$$\sup_{t\in[0,T]} \left\| \partial_{t}(\gs(t)-g(t)) \right\|_{\Hc^{s-7}} \leq C(R_{1})h,$$
%and
%$$\sup_{t\in[0,T]} \frac{\left\| \partial_{t}(\gs(t)-g(t)) \right\|_{\Hc^{s-3}}}{\langle t \rangle^{3}}  \leq C(R_{1})h,$$
%where $M$ is provided by proposition \ref{ordre1}, and $C(M)$ is a constant which depends only on $R_{1}.$
%\end{corollary}
%Corollary \ref{ordre1bis} will be  very useful in  the proof of the convergence of order two \eqref{convergence}.
%\bigskip
%
Let us now begin the proof of Proposition  \ref{ordre1}. 

\subsection{Estimate of $M_{T,s-3}(d)$}

 We shall first estimate $ d_{k}(t)$,  $k=\pm1.$
\begin{proposition}
\label{propzetaz}
Assuming that $\eta \in \Hc^{s+4}_{\nu}$ satisfies the assumption {\bf{(H)}}, there exists $C>0,$ $\eps_{0}>0$ and $h_{0}>0$ such that for every $h\in(0,h_{0}],$ every $\eps \in(0, \eps_{0}]$ and every $T>0,$ every solution of \eqref{eqdiff} such that $Q_{T,s-2, \nu-1}(\delta) \leq R_{1}h$ enjoys the estimate
$$M_{T,s-3}(d)\leq C(R) h ( 1 + (\eps  + \eps h) R_{1}^{2}),$$
where $C(R)$ is a number that depends only on $R$ (one can take $C(R)=(1+R+R^{2})C$).
\end{proposition}
\begin{proof}
From \eqref{eqdiff}, we obtain by taking the Fourier transform, integrating in time  and setting $\xi= nt$  that
\begin{equation}
\label{eqdn} d_{n}(t)=   \int_{0}^{t}K(n,t-\sigma) d_{n}(\sigma) \, \dd \sigma +   G_{n}(t) + H_{n}(t)
\end{equation}
where
$$ H_{n}(t) =  \int_{0}^t \hat{ \mathcal{R}}_{n}(\sigma, nt)\, \dd \sigma $$
 and
\begin{multline*}
 G_{n}(t) = 
  \eps \sum_{k=\pm1} \int_{0}^t  kp_{k} \Big( d_{k}(\sigma) \hat g_{n-k} (\sigma, nt- k\sigma) + \zeta_{k}(\sigma) \hat \delta_{n-k}(\sigma, 
   nt-k \sigma) \\
    -  d_{k}(\sigma) \hat \delta_{n-k}(\sigma, 
   nt-k \sigma) \Big)
   n(\sigma- t ) \, \dd\sigma. 
  \end{multline*}
The  kernel $K(k,t)$ is  still defined by \eqref{kernel}.\\
By using Lemma \ref{lemvolterra}, we find the estimate
\begin{equation}
\label{bordeaux}
M_{T, s-3}(d) \leq  C\left(
   M_{T, s-3}(G) + M_{T, s-3}(H)\right).
 \end{equation}
  To estimate $M_{T, s-3}(G)$, we proceed as in the proof of Theorem \ref{maintheod}.
   We first split $ G = G^1 + G^2$ where $G^1$ corresponds to the term $k= - n$ in the sum  and $G^2$ corresponds to $k= n$.
   For $G^1$, we obtain
\begin{multline*}
  |G^1_{n}| \leq  C\varepsilon \int_{0}^t  (t-\sigma) \Big( { 1 \over \langle \sigma \rangle^{s-3}} { \langle \sigma \rangle^3 \over \langle t + \sigma \rangle^s}
   M_{\sigma, s-3}(d) N_{\sigma, s, \nu}(g)  \\+
    {  1   \over \langle \sigma \rangle^{s- 1}} { \langle \sigma \rangle^3 \over \langle t + \sigma \rangle^{s-2}  } M_{\sigma, s- 1}(\zeta) N_{\sigma, s-2, \nu-1}(\delta)
     \\ + { 1 \over \langle \sigma \rangle^{s-3}} {\langle \sigma \rangle^{3} \over \langle  t + \sigma \rangle^{s-2}}  M_{\sigma, s-3}(d)N_{\sigma, s-2, \nu-1}(\delta)
      \Big) \dd\sigma,
 \end{multline*}
 and hence that
 $$  |G^1_{n}| \leq  C\eps  (R R_{1}h +h^2 R_{1}^2) { 1 \over \langle t \rangle^{s-2}}$$
  since  $s \geq 8$.
 To estimate $G^2_{n}$, we can again write
 $$ G^2_{n}= J^1_{n}+ J^2_{n}$$ with
 $$ J^1_{n}= \int_{0}^{t\over 2}j_{n}\, \dd \sigma, \quad J^2_{n}= \int_{t\over 2}^t j_{n}\, \dd \sigma$$
 with 
 $$
  j_{n}= 
  \eps  np_{n} \Big( d_{n}(\sigma) \hat g_{0} (\sigma, nt- n\sigma) + \zeta_{n}(\sigma) \hat \delta_{0}(\sigma, 
   nt-n \sigma) 
    -  d_{n}(\sigma) \hat \delta_{0}(\sigma, 
   nt-n \sigma) \Big)
   n(\sigma- t ).
 $$  
 As in the proof of Proposition \ref{propzeta}, we can prove  by using the same estimates as above that
 $$  |J_{n}^1| \leq  C \eps  (R R_{1}h + R_{1}^2 h^2) { 1 \over \langle t \rangle^{s-2}}.$$
  To estimate $J_{n}^2$, we also proceed as in the proof of Proposition \ref{propzeta} and  write
 \begin{multline*}
  |J^2_{n}| \leq C \eps  \int_{{t \over 2}}^t (t- \sigma) \Big( { 1 \over \langle \sigma \rangle^{s-3}}   { 1 \over \langle t-\sigma \rangle^{s-4}}  M_{\sigma, s-3}(d)
   \|g\|_{\mathcal{H}^{s-4}_{\nu}} +  { 1 \over \langle \sigma \rangle^{s-1}}   { 1 \over \langle t-\sigma \rangle^{s-6}}  M_{\sigma, s-1}(\zeta)
   \|\delta\|_{\mathcal{H}^{s-6}_{\nu-1}}  \\+  { 1 \over \langle \sigma \rangle^{s-3}}   { 1 \over \langle t-\sigma \rangle^{s-6}}  M_{\sigma, s-3}(d)
   \|\delta\|_{\mathcal{H}^{s-6}_{\nu-1}} \Big) \, \dd\sigma. 
   \end{multline*} 
   This yields, since $s \geq 8$, 
  $$ | J_{n}^2 |   \leq C \eps  (R R_{1}h +  R_{1}^2 h^2) { 1 \over \langle t \rangle^{s-3}}.$$
  We have thus proven that
  $$ M_{T, s-3}(G) \leq   C \eps (R R_{1}h + R_{1}^2 h^2).$$
  
  It remains to estimate $M_{T, s-3}(H)$. We shall prove that
  $$  M_{T, s-3}(H) \leq C(R) h.$$
  At first, we can write
  \begin{equation}
  \label{Hdec} H_{n}= H_{n}^L + \eps H_{n}^P
  \end{equation}
   with
$$ H_{n}^L =  \int_{0}^t
  np_{n}  \Big( \hat\gs_{n}(\sigma, n \sigma) \hat \eta_{0}(nt- n \sigma) (n\sigma- n t  ) - \hat  \gs_{n}(\sigma, n \ssf(\sigma)) \hat \eta_{0}(nt- n \ssf(\sigma)) (n\ssf(\sigma) -  nt ) \Big) \, \dd\sigma $$
  and
  \begin{multline*}
  H_{n}^P = \int_{0}^t  \sum_{k=\pm1}  kp_{k} \Big(  \hat\gs_{k}(\sigma, k \sigma) \hat \gs_{n-k} (\sigma, nt - k\sigma) (n\sigma- n t)
      \\ -    \hat\gs_{k}(\sigma, k \ssf(\sigma)) \hat \gs_{n-k} (\sigma, nt - k\ssf(\sigma)) (n \ssf(\sigma)- nt ) \Big) \dd\sigma.
  \end{multline*}
  We shall focus on the estimate of $H_{n}^P,$ the estimate of $H_{n}^L$  is easier to obtain since $\eta$ can be assumed as smooth as we need (here $\eta \in \Hc^{s+4}_{\nu}$ suffices).\\
  By Taylor expanding up to the first order, we have
  $$\left| H_{n}^{P} \right|\leq h \sum_{k=\pm1}p_{k}\int_{0}^{t}\sup_{\alpha+\beta \leq 1} \sup_{\xi\in [k\ssf(\sigma) ; k\sigma]}\left| \partial_{\xi}^{\alpha}\hat{\gs}_{k}(\sigma,\xi) \partial_{\xi}^{\beta}\hat{\gs}_{n-k}(\sigma,nt-\xi) \right|\left|n\sigma-nt\right|\dd \sigma.$$
As previously, we distinguish the cases $k=n$ and $k=-n.$\\
When $k=-n,$ we have by Lemma \ref{lememb}
\begin{multline*}
\int_{0}^{t}\sup_{\alpha+\beta \leq 1} \sup_{\xi\in [-n\ssf(\sigma); -n\sigma]}\left| \partial_{\xi}^{\alpha}\hat{\gs}_{-n}(\sigma,\xi) \partial_{\xi}^{\beta}\hat{\gs}_{2n}(\sigma,nt-\xi) \right|\left|\sigma-t\right|\dd \sigma 
\\ \leq C \int_{0}^{t}  \frac{\langle \sigma \rangle^{6}N_{T,s,\nu}(\gs)N_{T,s,\nu}(\gs)}{\langle \sigma \rangle^{s}\langle t+s \rangle^{s-1}}\dd \sigma \leq \frac{CR^{2}}{\langle t \rangle^{s-1}},
\end{multline*}
since $s\geq 8.$\\
When $k=n,$ we split as previously the integral into
\begin{multline*}
\int_{0}^{t}\sup_{\alpha+\beta \leq 1} \sup_{\xi\in [n\ssf(\sigma) ; n\sigma]}\left| \partial_{\xi}^{\alpha}\hat{\gs}_{n}(\sigma,\xi) \partial_{\xi}^{\beta}\hat{\gs}_{0}(\sigma,nt-\xi) \right|\left|n\sigma-nt\right|\dd \sigma 
\\= \int_{0}^{t/2}\sup_{\alpha+\beta \leq 1} \sup_{\xi\in [n\ssf(\sigma) ; n\sigma]}\left| \partial_{\xi}^{\alpha}\hat{\gs}_{n}(\sigma,\xi) \partial_{\xi}^{\beta}\hat{\gs}_{0}(\sigma,nt-\xi) \right|\left|n\sigma-nt\right|\dd \sigma
\\+ \int_{t/2}^{t}\sup_{\alpha+\beta \leq 1} \sup_{\xi\in [n\ssf(\sigma) ; n\sigma]}\left| \partial_{\xi}^{\alpha}\hat{\gs}_{n}(\sigma,\xi) \partial_{\xi}^{\beta}\hat{\gs}_{0}(\sigma,nt-\xi) \right|\left|n\sigma-nt\right|\dd \sigma.
\end{multline*}
Using the same estimates as above, we have for the first term
$$\left| \int_{0}^{t/2}\sup_{\alpha+\beta \leq 1} \sup_{\xi\in [n\ssf(\sigma) ; n\sigma]}\left| \partial_{\xi}^{\alpha}\hat{\gs}_{n}(\sigma,\xi) \partial_{\xi}^{\beta}\hat{\gs}_{0}(\sigma,nt-\xi) \right|\left|n\sigma-nt\right|\dd \sigma \right| \leq \frac{CR^{2}}{\langle t \rangle^{s-1}}.$$
For the second term, we rather use the quantity $\left\| \gs \right\|_{\Hc^{s-4}_{\nu}}$ instead of $N_{T,s,\nu}(\gs).$ We obtain the estimate
\begin{multline*}
\left| \int_{t/2}^{t}\sup_{\alpha+\beta \leq 1} \sup_{\xi\in [n\ssf(\sigma) ; n\sigma]}\left| \partial_{\xi}^{\alpha}\hat{\gs}_{n}(\sigma,\xi) \partial_{\xi}^{\beta}\hat{\gs}_{0}(\sigma,nt-\xi) \right|\left|n\sigma-nt\right|\dd \sigma \right| 
\\ \leq C \int_{t/2}^{t} \frac{\langle \sigma \rangle^{3}N_{T,s,\nu}(\gs) \left\| \gs(\sigma) \right\|_{\Hc^{s-4}_{\nu}}}{\langle \sigma \rangle^{s}\langle t-s \rangle^{s-5}}\dd \sigma \leq \frac{CR^{2}}{\langle t \rangle^{s-3}},
\end{multline*}
since $s\geq 8.$\\
Hence
$$M_{T,s-3}(H^{P})\leq ChR^{2},$$
and by the same arguments
$$M_{T,s-3}(H^{L})\leq ChR^{2}.$$
Therefore
$$M_{T,s-3}(H)\leq ChR^{2},$$
which concludes the proof of proposition \ref{propzetaz}.

%&+ \sup_{t\in[0,T]}\sup_{k=\pm1} \langle t \rangle ^{s-2} \left|G_{k}(t)\right| \\ 
%&+ \epsilon \sup_{t\in[0,T]}\sup_{k=\pm1} \langle t \rangle ^{s-2} \left| F_{k}(t) - H_{k}(t)\right|.
%\end{split}
%\end{equation}
%Let us now estimate the three quantities in the RHS. In the proof of proposition \ref{propzeta}, we already proved that there exists a constant $C$ such that
%\begin{equation}
%\label{estimG}
%\sup_{t\in[0,T]}\sup_{k=\pm1} \langle t \rangle ^{s-1} \left|G_{k}(t)\right|\leq Ch \sup_{t\in[0,T]}\sup_{k=\pm1} \langle t \rangle ^{s-1} \left|\zs_{k}(t)\right|\leq CR h,
%\end{equation}
%using Theorem \ref{maintheod}.\\
\end{proof}
\subsection{Estimate of $\left\| \delta  \right\|_{\Hc^{s-6}_{\nu-1}}$}
\begin{proposition}
\label{propgsg-5}
Assuming that $\eta \in \Hc^{s+4}$ satisfies the assumption {\bf{(H)}}, there exists $C>0,$ $\eps_{0}>0$ and $h_{0}>0$ such that for every $h\in(0,h_{0}],$ every $\eps \in (0,\eps_{0}]$ and every $T>0,$ every solution of \eqref{eqdiff} such that $Q_{T,s-2, \nu- 1}(\delta)\leq R_{1}h$  enjoys the estimate
$$ \sup_{t\in[0, T]}\left\|\delta(t) \right\|_{\Hc^{s-6}_{\nu-1}} \leq C(R) h ( 1 + (\eps  + \eps h) R_{1}^{2})( 1 + h R_{1})$$
where $C(R)$ is a number that depends only on $R$.
\end{proposition}
\begin{proof}
 By using the notation
  $\mathcal{L}_{t}[g] h=  \{ \phi(t, g), h\}, $
  we can rewrite \eqref{eqdiff} as
  \begin{equation}
  \label{eqdiff2}
  \partial_{t} \delta(t) = \mathcal{L}_{t}[\delta(t) ] (\eta - \eps \delta) + \eps \mathcal{L}_{t}[ \delta(t) ]g + \eps \mathcal{L}_{t}[g(t)] \delta + \mathcal{R}.
  \end{equation}
  Let $D$ be the linear operator defined as the Fourier multiplier by $k^{p} \xi^{q} \partial_{\xi}^{m}$ for $(m,p,q)\in \mathbbm{N}^{3d}$ such that 
$p+q\leq s-6,$ and $m\leq \nu -1.$
 From an energy estimate as in proof of Proposition \ref{propgs}, we obtain 
  that
  \begin{multline}
  \label{energyid}
  \left\| D\delta(t)\right\|_{L^{2}}^{2} = \int_{0}^t \Big( -\eps \langle [D, \mathcal{L}_{\sigma}[\delta(\sigma)]]\delta(\sigma), D \delta(\sigma) \rangle_{L^2}
   +  \eps \langle [D, \mathcal{L}_{\sigma}[g(\sigma)] ]\delta(\sigma), D \delta(\sigma) \rangle_{L^2}
    \\+    \langle D\left(\mathcal{L}_{\sigma}[\delta(\sigma)] (\eta + \eps  g(\sigma))\right), D \delta(\sigma) \rangle_{L^2}
     + \eps \langle D \mathcal{R}, D \delta(\sigma) \rangle_{L^2} \, \Big) \dd\sigma .
  \end{multline}
  By using Lemma \ref{lemcom}, we  thus obtain that
 \begin{multline*}
 \sup_{[0, T]}  \left\|\delta(t)\right\|_{\mathcal{H}^{s-6}_{\nu-1}} \leq C \int_{0}^T \Big(
   \eps (  m_{\sigma,s-5}(d(\sigma))  \| \delta (\sigma) \|_{\Hc^1_{\nu-1}} +  m_{\sigma, 2}(d(\sigma))  \| \delta (\sigma) \|_{\Hc^{s-6}_{\nu-1}})
    \\  + \eps (m_{\sigma,s-5}(\zeta(\sigma))  \| \delta(\sigma) \|_{\Hc^1_{\nu-1}} +   m_{\sigma, 2}(\zeta(\sigma)) \|\delta (\sigma)\|_{\Hc^{s-6}_{\nu-1}})
       \\ \quad \quad \quad +  m_{\sigma,s-5}(d(\sigma))  \|\eta + \eps  g(\sigma) \|_{\Hc^1_{\nu-1}}  + m_{\sigma, 2}(d(\sigma))\|\eta + \eps g(\sigma) \|_{\Hc^{s-5}_{\nu-1}} 
       +  \| D \mathcal{R}\|_{L^2} \Big) \, \dd \sigma.
  \end{multline*} 
   Next, we can use the fact that  $Q_{t, s, \nu}(g) \leq R$  and
    Proposition \ref{propzetaz}
    to obtain that 
      \begin{align*}
 \sup_{t\in[0, T]} \left\|\delta(t)\right\|_{\mathcal{H}^{s-6}_{\nu-1}} & \leq  C(R) h ( 1 + (\eps  + \eps h) R_{1}^{2})( 1 + h R_{1}) \int_{0}^T  {1 \over \langle \sigma \rangle^{2}}\, \dd \sigma
   + C \int_{0}^T \| \mathcal{R}(\sigma) \|_{\mathcal{H}^{s-6}_{\nu-1}}\, \dd\sigma \\
   & \leq  C(R) h ( 1 + (\eps  + \eps h) R_{1}^{2}) ( 1 + hR_{1}) +  C \int_{0}^T \| \mathcal{R}(\sigma) \|_{\mathcal{H}^{s-6}_{\nu-1}}\, \dd\sigma.
   \end{align*}
     It remains to estimate the last integral to conclude. By using the expression \eqref{Rfourier}, we get as in the proof of Lemma \ref{lemcom}
   that
   \begin{multline*}
   \| \mathcal{R}(\sigma) \|_{\mathcal{H}^{s-6}_{\nu-1}} \leq
    C  h \Big(  m^{(1)}_{\sigma, s-5}(\gs(\sigma)) ( \| \eta \|_{\Hc^1_{\nu}}  +  \| \gs(\sigma) \|_{\Hc^1_{\nu}} )
     +    m^{(1)}_{\sigma, 2}(\gs(\sigma)) (  \| \eta \|_{\Hc^{s-5}_{\nu}}  +  \| \gs(\sigma) \|_{\Hc^{s-5}_{\nu}} ) \Big) 
        \end{multline*}
     where we have set 
     $$  m^{(1)}_{\sigma, \gamma} (h) =\langle \sigma \rangle^{\gamma} \sup_{n= \pm1} \sup_{ \xi \in [n\sigma, n\ssf(\sigma)]}( |\hat h_{n} ( \xi)| +  |\partial_{\xi }\hat h_{n} ( \xi)|).$$
      Note that thanks to lemma   \ref{lememb}, we have that
 \begin{equation}
\label{embtildem}
 m^{(1)}_{\sigma, s-5}(\gs(\sigma)) \leq C  { \langle \sigma \rangle^{s-5 + 3} \over \langle \sigma \rangle^s} \| \gs(\sigma) \|_{\Hc^s_{\nu}}
  \leq { C \over \langle \sigma \rangle^2} R
\end{equation}     
since  $Q_{t, s,\nu}(\gs) \leq R$.
This yields 
    \begin{equation}
    \label{LieR} \int_{0}^t   \| \mathcal{R}(\sigma) \|_{\mathcal{H}^{s-6}_{\nu-1}} \leq CR h \int_{0}^t { 1 \over \langle \sigma \rangle^2}\, \dd\sigma
     \leq  CR h .
     \end{equation}
     Consequently, we get that  
\begin{equation}
\notag
\sup_{t\in[0,T]} \left\| \delta (t) \right\|_{\Hc^{s-6}_{\nu-1}} \leq C(R) h ( 1 + (\eps+\eps h)  R_{1}^{2})( 1 + h R_{1}).
\end{equation}
This ends the proof of Proposition \ref{propgsg-5}.
\end{proof}

\subsection{Estimate of $N_{T,s-2, \nu-1}(\delta)$}
\begin{proposition}
\label{propgsg}
Assuming that $\eta \in \Hc^{s+4}$ satisfies the assumption {\bf{(H)}}, there exists $C>0,$ $\eps_{0}>0$ and $h_{0}>0$ such that for every $h\in(0,h_{0}],$ every $\eps \in(0,\eps_{0}]$ and every $T>0,$ every solution of \eqref{eqdiff} such that $Q_{T,s-2, \nu-1}(\delta)\leq hR_{1}$ enjoys the estimate
$$N_{T,s-2, \nu-1}(\delta) \leq C(R) h ( 1 + (\eps+\eps h)  R_{1}^{2})( 1 + h R_{1})$$
where $C(R)$ is a number that depends only on $R$.
\end{proposition}
\begin{proof}
We proceed as in the previous proof.  Let $D$ be the linear operator defined as the Fourier multiplier by $k^{p} \xi^{q} \partial_{\xi}^{m}$ for $(m,p,q)\in \mathbbm{N}^{3d}$ such that 
$p+q\leq s-2,$ and $m\leq \nu -1.$ Using \eqref{energyid} with this operator $D,$ and Lemma \ref{lemcom}, we  obtain that
 \begin{multline*}
 \left\|\delta(t)\right\|_{\mathcal{H}^{s-2}_{\nu-1}}^{2} \leq C \int_{0}^t \Big(
   \eps (  m_{\sigma,s-1}(d(\sigma))  \| \delta (\sigma) \|_{\Hc^1_{\nu-1}}\| \delta (\sigma) \|_{\Hc^{s-2}_{\nu-1}} +  m_{\sigma, 2}(d(\sigma))  \| \delta (\sigma) \|_{\Hc^{s-2}_{\nu-1}}^{2})
    \\  + \eps (m_{\sigma,s-1}(\zeta(\sigma))  \| \delta(\sigma) \|_{\Hc^1_{\nu-1}} \| \delta (\sigma) \|_{\Hc^{s-2}_{\nu-1}} +   m_{\sigma, 2}(\zeta(\sigma)) \|\delta (\sigma)\|_{\Hc^{s-2}_{\nu-1}}^{2})
       \\ \quad \quad \quad +  m_{\sigma,s-1}(d(\sigma))  \|\eta + \eps  g(\sigma) \|_{\Hc^1_{\nu-1}}\| \delta (\sigma) \|_{\Hc^{s-2}_{\nu-1}} 
       \\ + m_{\sigma, 2}(d(\sigma))\|\eta + \eps g(\sigma) \|_{\Hc^{s-1} _{\nu-1}}  \| \delta (\sigma) \|_{\Hc^{s-2}_{\nu-1}}
      +  \| D \mathcal{R}\|_{L^2} \| \delta (\sigma) \|_{\Hc^{s-2}_{\nu-1}} \Big) \, \dd \sigma.
  \end{multline*} 
  Using now  $Q_{t, s, \nu}(g) \leq R$  and
    Proposition \ref{propzetaz},
    we obtain that 
\begin{multline*}
\left\|\delta(t)\right\|_{\mathcal{H}^{s-2}_{\nu-1}}^{2} 
    \leq  C\langle t \rangle^{6} h^{2} C(R)^{2}( 1 + (\eps  + \eps h) R_{1}^{2})^{2} ( 1 + hR_{1})^{2} + 
    \\  \langle t \rangle^{3} hC(R) (1 + (\eps  + \eps h) R_{1}^{2}) \int_{0}^t \| \mathcal{R}(\sigma) \|_{\mathcal{H}^{s-2}_{\nu-1}}\, \dd\sigma.
\end{multline*}
     It remains to estimate the last integral to conclude. By using the expression \eqref{Rfourier}, we get as in the proof of Lemma \ref{lemcom}
   that
   \begin{multline*}
   \| \mathcal{R}(\sigma) \|_{\mathcal{H}^{s-2}_{\nu-1}} \leq
    C  h \Big(  m^{(1)}_{\sigma, s-1}(\gs(\sigma)) ( \| \eta \|_{\Hc^1_{\nu}}  +  \| \gs(\sigma) \|_{\Hc^1_{\nu}} )
     +    m^{(1)}_{\sigma, 2}(\gs(\sigma)) (  \| \eta \|_{\Hc^{s-1}_{\nu}}  +  \| \gs(\sigma) \|_{\Hc^{s-1}_{\nu}} ) \Big) 
        \end{multline*}
     with 
     $$  m^{(1)}_{\sigma, \gamma} (h) =\langle \sigma \rangle^{\gamma} \sup_{n= \pm1} \sup_{ \xi \in [n\sigma, n\ssf(\sigma)]}( |\hat h_{n} ( \xi)| +  |\partial_{\xi }\hat h_{n} ( \xi)|).$$
      As in the previous proof we have that
 \begin{equation}
\label{embtildem2}
 m^{(1)}_{\sigma, s-1}(\gs(\sigma)) \leq C  { \langle \sigma \rangle^{s-1 + 3} \over \langle \sigma \rangle^s} \| \gs(\sigma) \|_{\Hc^s_{\nu}}
  \leq C  \langle \sigma \rangle^2 R
\end{equation}     
since  $Q_{t, s,\nu}(\gs) \leq R$.
This yields 
    \begin{equation}
    \label{LieR2} \int_{0}^t   \| \mathcal{R}(\sigma) \|_{\mathcal{H}^{s-2}_{\nu-1}} \leq C(R)h \langle t \rangle^{3}.
     \end{equation}
     Collecting all the above estimates, we obtain  
\begin{equation}
\notag
\sup_{t\in[0,T]} \langle t \rangle^{-3} \left\| \delta (t) \right\|_{\Hc^{s-2}_{\nu-1}} \leq C(R) h ( 1 + (\eps+\eps h)  R_{1}^{2})( 1 + h R_{1}).
\end{equation}
This ends the proof of Proposition \ref{propgsg}.
\end{proof}

\subsection{Proof of Proposition \ref{ordre1}}
From Propositions \ref{propzetaz}, \ref{propgsg-5} and \ref{propgsg}, we have the estimate
$$Q_{T,s-2, \nu-1}(\delta)\leq C(R) h( 1+ (\eps + \eps h)R_{1}^2)( 1+ hR_{1}),$$
under the assumption $Q_{T,s-2, \nu-1}(\delta) \leq h R_1.$ Choosing $R_{1}> C(R),$ we have
$$ C(R) h( 1+ (\eps + \eps h)R_{1}^2)( 1+ hR_{1})<R_{1}h,$$
if $\eps$ and $h$ are   small enough. Hence, by an usual continuation argument, the estimate $Q_{T,s-2, \nu-1}(\delta)\leq R_{1}h$ is valid for all $T>0,$ thus proving Proposition \ref{ordre1}  and with it the convergence estimate \ref{convergence1}.

%One can then easily get Corollary  \ref{ordre1bis} by using the equation for $\gs - g$, the error estimate \ref{convergence1} 
% and the estimates of Lemma \ref{lemcom}.
%Let us now prove corollary \ref{ordre1bis}. In view of \eqref{diff1}, we have to bound
%$$\left\| \epsilon \mathcal{L}_{t}^{1}[\gs](\gs(t)-g(t)) \right\|_{\Hc^{s-6}}, \hspace{4mm} \left\| \left[\mathcal{L}_{t}^{1}[\gs]-\mathcal{L}_{t}^{2}[g]\right](\eta+ \epsilon g(t)) \right\|_{\Hc^{s-6}},$$
%and also
%$$\left\| \epsilon \mathcal{L}_{t}^{1}[\gs](\gs(t)-g(t)) \right\|_{\Hc^{s-2}}, \hspace{4mm} \left\| \left[\mathcal{L}_{t}^{1}[\gs]-\mathcal{L}_{t}^{2}[g]\right](\eta+ \epsilon g(t)) \right\|_{\Hc^{s-2}}.$$
%Using again \eqref{com2}, and the same arguments as in the proof of proposition \ref{ordre1}, we find that
%\begin{equation}
%\notag
%\begin{split}
%&\left\| \epsilon \mathcal{L}_{t}^{1}[\gs](\gs(t)-g(t)) \right\|_{\Hc^{s-6}} \leq \frac{\epsilon C(R) M_{T,s}(\zs)}{\langle t \rangle ^{2}} \left\| \gs(t)-g(t) \right\|_{\Hc^{s-5}},\\
%&\left\| \epsilon \mathcal{L}_{t}^{1}[\gs](\gs(t)-g(t)) \right\|_{\Hc^{s-2}} \leq \epsilon C(R) M_{T,s}(\zs) \left\| \gs(t)-g(t) \right\|_{\Hc^{s-1}},\\
%&\left\| \left[\mathcal{L}_{t}^{1}[\gs]-\mathcal{L}_{t}^{2}[g]\right](\eta+ \epsilon g(t)) \right\|_{\Hc^{s-6}}\leq \frac{C(R) M_{T,s-2}(\zs - \zeta)}{\langle t \rangle^{2}}\\
%&\left\| \left[\mathcal{L}_{t}^{1}[\gs]-\mathcal{L}_{t}^{2}[g]\right](\eta+ \epsilon g(t)) \right\|_{\Hc^{s-2}}\leq C(R) M_{T,s-2}(\zs - \zeta).
%\end{split}
%\end{equation}
%Hence it is clear that proposition \ref{ordre1} implies corollary \ref{ordre1bis}.
%
\section{Proof of the convergence estimate \eqref{convergence} for Strang splitting}

From now on, we only consider the case of Strang splitting \eqref{strang}, and thus $\ssf(t)$ is given by formula \eqref{prince}. Proposition \ref{ordre1} implies that, up to the loss of two  derivatives, $\gs(t,x,v)$ is an approximation of order one (with respect to $h$) of the exact solution of the Vlasov-HMF equation $g(t,x,v).$\\
Getting the rate of order $2$ brings technical complications  in order to take advantage of the cancellations provided by the midpoint rule.\\
In view of theorem \ref{maintheod} and the main result of \cite{RF}, we can assume that for all $\alpha \in \{0,1,2\},$ $Q_{T,s,\nu-\alpha}(\gs)$ and $Q_{T,s,\nu-\alpha}(g)$ are bounded by the same constant $R>0,$ provided that $\nu >5/2.$
% we are forced to consider the difference between $\gs(t)$ and $g(t)$ only at times $t=nh,$ that is when $\gs(t,x,v)$ coincides with the sequence given by the Strang splitting method. Thus the two other factors of $Q_{T,s-1}(\gs-g)$ also have to be modified.\\
%Let $n\in \mathbbm{N}.$ We set
%$$\tilde{M}_{n,s}(\zs-\zeta) = \sup_{t\in[0,nh]}\sup_{k=\pm1} \langle t \rangle^{s-4} \left| \hat{\gs}_{k}(t,kt)-\hat{g}_{k}(t,kt)\right|,$$
%$$\tilde{N}_{n,s}(\gs-g) = \frac{\left\| g^{n}-g(nh)\right\|_{\Hc^{s-3}}}{\langle nh \rangle ^{3} },$$
%and
%$$\tilde{Q}_{n,s}(\gs-g)= \tilde{M}_{n,s}(\zs-\zeta) + \tilde{N}_{n,s}(\gs-g) + \left\| g^{n}-g(nh)\right\|_{\Hc^{s-7}} .$$
To prove the result, we proceed as before and start from the equation \eqref{eqdiff} on the error term $\delta(t)$. 
From now on, using the weighted norms defined in \eqref{eq:NetM}, we consider the quantity:
$$Q_{T,s-3,v-2}(\delta)=M_{T,s-4}(d)+N_{T,s-3,v-2}(\delta) + \sup_{t\in[0,T]}\left\| \delta(t) \right\|_{\Hc^{s-7}_{\nu-2}}.$$
The convergence result \eqref{convergence} will be a consequence of the following proposition:
\begin{proposition}
\label{ordre2}
Let us fix $s\geq 9$ and $\nu >5/2.$
Assuming that $\eta \in \Hc^{s+4}_{\nu}$ satisfies the assumption {\bf{(H)}}, there exists $R_{2}>0$,  $h_{0}>0$ and $\eps_{0}>0$  such that for every $h\in(0,h_{0}]$, every  
$ \eps \in(0, \eps_{0}]$  and  every $T>0$, the solution of \eqref{eqdiff} satisfies the estimate 
$$Q_{T,s-3, \nu-2}(\delta)\leq R_{2} h^{2}.$$
\end{proposition}

The crucial point of the proof of proposition \ref{ordre2} is the cancellation provided by the midpoint rule. We can summarize it by the following easy  lemma:
\begin{lemma}
\label{midpoint}
For $t\in \mathbbm{R}$ and a fixed $h>0,$ let $\ssf(t)=\ssf_{h}(t)$ be given by formula \eqref{prince}. Then, for all $n\in \mathbbm{N},$
$$\int_{nh}^{(n+1)h} (\sigma-\ssf(\sigma)) \dd \sigma =0.$$
\end{lemma}

\subsection{Estimate of $M_{T,s-4}(d) $}
\begin{proposition}
\label{propzetaz2}
Assuming that $\eta\in \Hc^{s+4}_{\nu}$ satisfies the assumption {\bf{(H)}}, there exists $C>0,$ $\eps_{0}>0$ and $h_{0}$ such that for every $h\in(0,h_{0}],$ every $\eps \in(0,\eps_{0}]$
 and every $T>0,$ every solution of \eqref{eqdiff} such that  $Q_{T,s-3, \nu-2}(\delta) \leq R_{2}h^2$ enjoys the estimate
$$M_{T,s-4}(d)\leq  C(R) h^2( 1 +  (\eps + \eps h^2) R_{2}^2)$$
\end{proposition}
\begin{proof}
From Equation \eqref{eqdn} and Lemma \ref{lemvolterra}, we still get that
$$ M_{T, s-4}(d) \leq  C\left(
  M_{T, s-4}(G) +  M_{T, s-3}(H)\right).$$
  By using the same arguments as in the proof of Proposition \ref{propzetaz}, we can  easily obtain that
  $$ M_{T, s-4}(G) \leq  C\eps(  R  R_{2} h^2 + h^4 R_{2}^2).$$
 It thus remains the   estimate  of   $M_{T, s-4}(H)$ that requires a refined  analysis of the cancellations in the integral.
   By using again the decomposition \eqref{Hdec}, we shall focus on the term $H^P$ that is more difficult.
    By Taylor expanding up to second order, we find that  
\begin{multline*}
 H_{n}^P = \int_{0}^t \sum_{k=\pm1} k^2p_{k}  (\sigma - \ssf(\sigma))  \Big(\partial_{\xi}  \hat \gs_{k}( \ssf(\sigma), k \ssf(\sigma)) \hat \gs_{n-k} (\ssf(\sigma),  nt - k \ssf(\sigma))
  \\+   \hat \gs_{k}( \ssf(\sigma), k \ssf(\sigma))\partial_{\xi} \hat \gs_{n-k} (\ssf(\sigma),  nt - k \ssf(\sigma))\Big)
  ( n\ssf(\sigma) - nt) \, \dd\sigma 
   +  K_{n}
\end{multline*}
 where
$$ 
|K_{n }| \leq  h^2 \sum_{k=\pm1} p_{k}\int_{0}^t  \sup_{ \substack{\alpha + \beta_{1} \leq 2, \, \alpha \neq 2 \\  \alpha + \beta_{2} \leq 2, \, \alpha \neq 2} }\, \, \sup_{ \tau, \, \xi  \in [ k \sigma,  k\ssf(\sigma)]} | \partial_{t}^\alpha \partial_{\xi}^{\beta_{1}} \hat \gs_{k}(\tau, \xi) |
 | \partial_{t}^\alpha \partial_{\xi}^{\beta_{2}} \hat \gs_{n-k}(\tau, nt-\xi)| | n\sigma - nt|\, \dd \sigma.
$$
To estimate the remainder,  $K_{n}$, we can again distinguish the cases $k= n$ and $k= -n$.
 Since $Q_{T, s, \nu}(\gs) \leq R$, by using the equation \eqref{julienclerc} and Lemma \ref{lemcom}, we also  have that
  $$ \| \partial_{t} \gs (t) \|_{\Hc^{s-1}_{\nu}} \leq C(R) \langle t \rangle^2.$$ 
  Since $\nu>5/2$, we can use Lemma \ref{lememb} to obtain, by similar arguments to the ones used in the proof of Proposition \ref{propzetaz}, that
 $$
 |K_{n}| \leq  h^2 C(R) \Big(\int_{0}^t { 1\over \langle \sigma \rangle^{s- 4}} { 1 \over \langle t +  \sigma \rangle^{ s- 3}} (t - \sigma) \,d\sigma
   +   \int_{t/2}^t   { 1\over \langle \sigma \rangle^{s- 4}}  { 1 \over \langle t -  \sigma \rangle^{ s- 3}} (t - \sigma) \,\dd\sigma \Big)
 $$
 and hence that
 $$ |K_{n}| \leq { 1 \over \langle t \rangle^{s-4}}C(R) h^2.$$
  To estimate the main term in $H_{n}^P$,   assuming that $Nh \leq T \leq (N+1)h$ for some $N$, we can split the time integral into
  $$ \int_{0}^t = \sum_{j=0}^{N-1} \int_{jh}^{(j+1)h} + \int_{Nh}^t$$
   and we observe that all the integrals   $\int_{jh}^{(j+1)h} $ vanish due to the symmetry of $\sigma - \ssf(\sigma)$ (see Lemma \ref{midpoint}).
    We thus obtain that
 $$ |H_{n}^P| \leq  { 1 \over \langle t \rangle^{s-4}}C(R) \int_{Nh}^t  | t - \ssf(\sigma)| \, d\sigma + 
   { 1 \over \langle t \rangle^{s-4}}C(R) h^2 \leq   { 1 \over \langle t \rangle^{s-4}}C(R) h^2.$$
   
    From the same argument (slightly easier since $\eta$ does not depend on time and is smoother), we can also prove that
    $$  |H_{n}^L | \leq   { 1 \over \langle t \rangle^{s-4}}C(R) h^2.$$ 
   
   Consequently, by collecting the previous estimates, we find that
   $$  M_{T, s-4}(G) \leq  C(R) h^2(  1 +  (\eps + \eps h^2) R_{2}^2).$$ 
   This ends the proof.
   \end{proof}

\subsection{Estimate of $\| \delta\|_{\Hc^{s-7}_{\nu-2}}$}
\begin{proposition}
\label{propgsg-7}
Assuming $\eta\in \Hc^{s+4}_{\nu}$ satisfies the assumption {\bf{(H)}}, there exists $C>0,$ $\eps_{0}$ and $h_{0}$ such that for every $h\in(0,h_{0}],$ every $\eps \in(0,\eps_{0}]$ and every $T>0,$ every solution of \eqref{eqdiff} such that  $Q_{T, s-3, \nu-2}(\delta) \leq h^2 R_{2}$ enjoys the estimate 
$$ \sup_{[0, T]}\|\delta \|_{\Hc^{s-7}_{\nu-2}}\leq   C(R) h^2(1 +  (\eps + \eps h^2) R_{2}^2) ( 1 + h^2 R_{2}).$$
\end{proposition}

\begin{proof}
We can start as in the proof of Proposition \ref{propgsg-5}.  By using the energy identity \eqref{energyid}  with $D$ the Fourier multiplier
 by $k^{p} \xi^{q} \partial_{\xi}^{m}$ for $(m,p,q)\in \mathbbm{N}^{3d}$ such that 
$p+q\leq s-7,$ and $m\leq \nu -2$
and Lemma \ref{lemcom}, we first obtain that
\begin{multline*}
  \left\| D \delta(t)\right\|_{L^2}^2 \leq C \int_{0}^t  \|D \delta (t)\|_{L^2} \Big(
   \eps (  m_{\sigma,s-6}(d(\sigma))  \| \delta (\sigma) \|_{\Hc^1_{\nu-2}} +  m_{\sigma, 2}(d(\sigma))  \| \delta (\sigma) \|_{\Hc^{s-7}_{\nu-2}})
    \\  + \eps (m_{\sigma,s-6}(\zeta(\sigma))  \| \delta(\sigma) \|_{\Hc^1_{\nu-2}} +   m_{\sigma, 2}(\zeta(\sigma)) \|\delta (\sigma)\|_{\Hc^{s-7}_{\nu-2}})
       \\ \quad \quad \quad +  m_{\sigma,s-6}(d(\sigma))  \|\eta + \eps  g(\sigma) \|_{\Hc^1_{\nu-2}}  + m_{\sigma, 2}(d(\sigma))\|\eta + \eps g(\sigma) \|_{\Hc^{s-6}_{\nu-2}}\Big) \, \dd\sigma  \\+  \left | \int_{0}^t 
        \langle D \mathcal{R}, D \delta(\sigma) \rangle_{L^2} \, \Big) \dd\sigma \right|.
  \end{multline*} 
  This yields 
   \begin{align}
   \nonumber
  \left\| D\delta(t)\right\|_{L^2}^2 & \leq  C(R) h^2 ( 1 +  ( \eps + \eps h^2)  R_{2}^2)( 1 + h^2 R_{2}) \int_{0}^t  {\| D \delta(\sigma) \|_{L^2} \over \langle \sigma \rangle^{2}}\, \dd \sigma
   +  \left | \int_{0}^t 
        \langle D \mathcal{R}, D \delta(\sigma) \rangle_{L^2} \, \dd\sigma \right| \\
 \label{strangfin1}  & \leq  C(R) h^2( 1  +  (\eps + \eps h^2)  R_{2}^2) ( 1 + h^2R_{2}) \sup_{\sigma\in[0, t]} \| \delta(\sigma) \|_{\Hc^{s-7}_{\nu-2}} + \left | \int_{0}^t 
        \langle D \mathcal{R}(\sigma), D \delta(\sigma) \rangle_{L^2} \, \dd\sigma \right| .
   \end{align}
   The main difficulty is now  to use the cancellation in the midpoint quadrature rule in order to estimate the last integral.
    Let us define
     $$ I (t)  =  \left | \int_{0}^t  \langle D \mathcal{R}(\sigma), D \delta(\sigma) \rangle_{L^2} \, \dd\sigma \right |. $$
      By Taylor expanding, integrating by parts once  and using the estimate \eqref{LieR}, we  first write that
    \begin{align*}
      I (t)  \leq   & \left | \int_{0}^t  \langle D \mathcal{R}(\sigma), D \delta(\ssf(\sigma)) \rangle_{L^2} \, \dd\sigma \right |
     + h \sup_{\sigma\in [0, T]} \| \partial_{t} \delta(\sigma) \|_{\Hc^{s-8}_{\nu-2}} \sup \int_{0}^t \| \mathcal R \|_{\Hc^{s-6}_{\nu- 2}}  \, \dd\sigma \\
     \leq   & \left | \int_{0}^t  \langle D \mathcal{R}(\sigma), D \delta(\ssf(\sigma)) \rangle_{L^2} \, \dd\sigma \right |
     + h^2 C(R) \sup_{\sigma\in [0, T]} \| \partial_{t} \delta (\sigma)\|_{\Hc^{s-8}_{\nu-2}}.
     \end{align*}
     By using the equation \eqref{eqdiff2} and Lemma \ref{lemcom}, we have that
    $$  \| \partial_{t} \delta (t) \|_{\Hc^{s-8}_{\nu-2}} 
     \leq  C m_{s-7}(d(t)) \| \eta\|_{\Hc^{s-7}_{\nu-2}} + \eps \big(  m_{s-7}(d(t)) ( \| \delta \|_{\Hc^{s-7}_{\nu-2}}  +   \| g \|_{\Hc^{s-7}_{\nu-2}})
      + m_{s-7}(\zeta) \| \delta \|_{\Hc^{s-7}_{\nu-2}} \big)$$
       and thus by using Proposition \ref{propzetaz2} we find 
       $$ \| \partial_{t} \delta (t) \|_{\Hc^{s-8}_{\nu-2}}  \leq
        C(R)h^2(1 + (\eps + \eps h^2) R_{2}^2) +  \eps  h^4 R_{2}^2.$$ 
      This yields
 $$ I (t)  \leq   \left | \int_{0}^t  \langle D \mathcal{R}(\sigma), D \delta(\ssf(\sigma)) \rangle_{L^2} \, \dd\sigma \right |
     +   C(R)h^4( 1 +  (\eps + \eps h^2) R_{2}^2).
  $$
     Next, by using the definition of $\mathcal{R}$ provided in \eqref{Rfourier}, we  observe that  by Taylor expanding in time and
      in the $\xi$ variable, we can  write
\begin{equation}
\label{Rdec}
\mathcal{R}(t,x,v)= \mathcal{R}^{1}(t,x,v) + \mathcal{R}^2(t,x,v) 
\end{equation}
      where $\mathcal{R}^1$ is given in Fourier by 
    \begin{multline*}
    \hat{\mathcal{R}}^1_{n}(t,\xi)= 
        (t- \ssf(t)) \Big( np_{n}\big( \partial_{\xi}\hat \gs_{n}(\ssf(t), n \ssf(t)) \hat \eta(\xi - n\ssf(t)) + \hat  \gs_{n}(\ssf(t), n \ssf(t))  \partial_{\xi} \hat \eta(\xi - n\ssf(t))
\big)(n \ssf  (t) - \xi )
 \\ +
 \eps \sum_{k=\pm1} k p_{k} ( \partial_{\xi} \hat \gs_{k}(\ssf(t), n \ssf(t))  \hat \gs_{n-k}(t, \xi- k \ssf(t)) +   \hat \gs_{k}(\ssf(t), n \ssf(t))  \partial_{\xi}\hat \gs_{n-k}(\ssf(t), \xi- k \ssf(t))) (n \ssf (t) - \xi) \Big)
     \end{multline*}
  and the remainder $\mathcal{R}^2$ is such that 
 $$
  \| \mathcal {R}^2 (t) \|_{\Hc^{s-7}_{\nu-2}}
   \leq C  h^2\Big( m^{(2)}_{t,s-6}(\gs) \big( \| \eta \|_{\Hc^1_{\nu}} +  \eps \|\gs \|_{\Hc^1_{\nu}} \big)+  m^{(2)}_{t,2}(\gs)  \big( \| \eta \|_{\Hc^{s-6}_{\nu}}
    +   \eps  \| \gs \|_{\Hc^{s-6}_{\nu}}\big) \Big)
 $$
     with 
     $$  m^{(2)}_{t, \gamma} (h) =\langle t \rangle^{\gamma} \sup_{n= \pm1} \sup_{ \xi \in [nt, n\ssf(t)]} \sum_{\alpha + \beta \leq 2, \, \alpha \neq 2} |\partial_{t}^\alpha \partial_{\xi}^\beta\hat h_{n} ( \xi)| .$$
     Since $\nu>5/2$, we have that 
     $$  m^{(2)}_{t,s-6}(\gs) \leq C { 1 \over \langle t \rangle^2} \| \gs \|_{\Hc^{s-4}_{\nu}} +  { 1 \over \langle t \rangle^{2}}  \| \partial_{ t}\gs \|_{\Hc^{s-4}_{\nu}} $$
      and hence thanks to the equation \eqref{julienclerc} and  Theorem \ref{maintheod}, we obtain that
      $$ m^{(2)}_{t,s-6}(\gs) \leq { C(R) \over \langle t \rangle^2}.$$
      This yields
     $$  \| \mathcal {R}^2 (t) \|_{\Hc^{s-7}_{\nu-2}}  \leq { C(R) \over \langle t \rangle^2},$$
     and hence 
    \begin{equation}
    \label{strangI}
    I (t)  \leq   \left | \int_{0}^t  \langle D \mathcal{R}^1, D \delta(\ssf(\sigma)) \rangle_{L^2} \, \dd\sigma \right |
     +  C(R) h^4( 1 +  (\eps + \eps h^2) R_{2}^2) + C(R) h^2  \sup_{\sigma \in[0, t]} \| \delta(\sigma) \|_{\Hc^{s-7}_{\nu-2}}.
     \end{equation}
      Thanks to the definition of $\mathcal{R}^1$,   assuming that $Nh \leq t \leq (N+1)h$, we obtain that
      $$  \left | \int_{0}^t  \langle D \mathcal{R}^1(\sigma), D \delta(\ssf(\sigma)) \rangle_{L^2} \, d\sigma \right | 
      =   \left | \int_{Nh}^t  \langle D \mathcal{R}^1(\sigma), D \delta(\ssf(\sigma)) \rangle_{L^2} \, d\sigma \right |
       \leq C(R) h^2  \sup_{\sigma \in[0, t]} \| \delta(\sigma) \|_{\Hc^{s-7}_{\nu-2}}.$$
       Here we have also used the cancellation argument of Lemma \ref{midpoint}.  
       Consequently, from \eqref{strangfin1}, \eqref{strangI} and the above estimate, we obtain that
      \begin{multline*} \sup_{\sigma\in[0, T]} \| \delta(\sigma) \|_{\Hc^{s-7}_{\nu-2}}^2 \leq 
        C(R)h^4( 1 +  (\eps + \eps h^2) R_{2}^2) ( 1 + h^2R_{2}) \\
        + C(R) h^2( 1  +  (\eps + \eps h^2)  R_{2}^2) ( 1 + h^2R_{2}) \sup_{\sigma\in[0, T]} \| \delta(\sigma) \|_{\Hc^{s-7}_{\nu-2}}.
        \end{multline*}
         By using the Young inequality this yields
         $$   \sup_{\sigma\in[0, T]} \| \delta(\sigma) \|_{\Hc^{s-7}_{\nu-2}}
          \leq    C(R)h^2( 1 +  (\eps + \eps h^2) R_{2}^2) ( 1 + h^2 R_{2}).
         $$
           This ends the  proof.

\end{proof}

\subsection{Estimate of $N_{T, s-3, \nu-2}(\delta)$}
\begin{proposition}
\label{propgsg2}
Assuming $\eta\in \Hc^{s+4}$ satisfies the assumption {\bf{(H)}}, there exists $C>0,$ $\eps_{0} >0,$ and $h_{0}$ such that for every $h\in(0,h_{0}],$ every $\eps \in (0,\eps_{0}]$ and every $T>0,$ every solution of \eqref{eqdiff} such that $Q_{T, s-3, \nu-2}(\delta) \leq h^2 R_{2}$ enjoys the estimate  
$$ N_{T, s-3, \nu-2}(\delta) \leq  C(R)h^2(1 +  (\eps + \eps h^2) R_{2}^2) ( 1 + h^2 R_{2}).$$
\end{proposition}

\begin{proof}

We can start once more as in the proof of Proposition \ref{propgsg-5}.  By using the energy identity \eqref{energyid}  with $D$ the Fourier multiplier
 by $k^{p} \xi^{q} \partial_{\xi}^{m}$ for $(m,p,q)\in \mathbbm{N}^{3d}$ such that 
$p+q\leq s-3,$ and $m\leq \nu -2$
and Lemma \ref{lemcom}, we first obtain that
\begin{multline*}
  \left\| D \delta(t)\right\|_{L^2}^2 \leq C \int_{0}^t  \|D \delta (t)\|_{L^2} \Big(
   \eps (  m_{\sigma,s-2}(d(\sigma))  \| \delta (\sigma) \|_{\Hc^1_{\nu-2}} +  m_{\sigma, 2}(d(\sigma))  \| \delta (\sigma) \|_{\Hc^{s-3}_{\nu-2}})
    \\  + \eps (m_{\sigma,s-2}(\zeta(\sigma))  \| \delta(\sigma) \|_{\Hc^1_{\nu-2}} +   m_{\sigma, 2}(\zeta(\sigma)) \|\delta (\sigma)\|_{\Hc^{s-3}_{\nu-2}})
       \\ \quad \quad \quad +  m_{\sigma,s-2}(d(\sigma))  \|\eta + \eps  g(\sigma) \|_{\Hc^1_{\nu-2}}  + m_{\sigma, 2}(d(\sigma))\|\eta + \eps g(\sigma) \|_{\Hc^{s-3}_{\nu-2}}\Big) \, \dd\sigma  \\+  \left | \int_{0}^t 
        \langle D \mathcal{R}, D \delta(\sigma) \rangle_{L^2} \, \Big) \dd\sigma \right|.
  \end{multline*} 
  This yields 
   \begin{align}
   \nonumber
  \left\| D\delta(t)\right\|_{L^2}^2 & \leq  C(R)^{2} h^2 ( 1+  ( \eps + \eps h^2)  R_{2}^2)( 1 + h^2 R_{2}) \int_{0}^t \langle \sigma \rangle^{2}  \| D \delta(\sigma) \|_{L^2} \, \dd \sigma
   +  \left | \int_{0}^t 
        \langle D \mathcal{R}, D \delta(\sigma) \rangle_{L^2} \, \dd\sigma \right| \\
 \label{strangfin2}  & \leq  C(R)^{2} \langle t \rangle^{6} h^4( 1  +  (\eps + \eps h^2)  R_{2}^2)^{2} ( 1 + h^2R_{2})^{2}+ \left | \int_{0}^t 
        \langle D \mathcal{R}(\sigma), D \delta(\sigma) \rangle_{L^2} \, \dd\sigma \right| .
   \end{align}
  As in the previous proof we consider
     $$ I (t)  =  \left | \int_{0}^t  \langle D \mathcal{R}(\sigma), D \delta(\sigma) \rangle_{L^2} \, \dd\sigma \right |, $$
      and obtain as previously
    \begin{align*}
      I (t)  \leq   & \left | \int_{0}^t  \langle D \mathcal{R}(\sigma), D \delta(\ssf(\sigma)) \rangle_{L^2} \, \dd\sigma \right |
     + h \sup_{\sigma\in [0, T]} \| \partial_{t} \delta(\sigma) \|_{\Hc^{s-4}_{\nu-2}} \sup \int_{0}^t \| \mathcal R \|_{\Hc^{s-2}_{\nu- 2}}  \, \dd\sigma \\
     \leq   & \left | \int_{0}^t  \langle D \mathcal{R}(\sigma), D \delta(\ssf(\sigma)) \rangle_{L^2} \, \dd\sigma \right |
     + h^2 C(R)\langle t \rangle^{3} \sup_{\sigma\in [0, T]} \| \partial_{t} \delta (\sigma)\|_{\Hc^{s-4}_{\nu-2}}.
     \end{align*}
     Here we have used estimate \eqref{LieR2}. 
     By using the equation \eqref{eqdiff2} and Lemma \ref{lemcom}, we have that
    $$  \| \partial_{t} \delta (t) \|_{\Hc^{s-4}_{\nu-2}} 
     \leq  C m_{s-3}(d(t)) \| \eta\|_{\Hc^{s-3}_{\nu-2}} + \eps \big(  m_{s-3}(d(t)) ( \| \delta \|_{\Hc^{s-3}_{\nu-2}}  +   \| g \|_{\Hc^{s-3}_{\nu-2}})
      + m_{s-3}(\zeta) \| \delta \|_{\Hc^{s-3}_{\nu-2}} \big)$$
       and thus by using Proposition \ref{propzetaz2} we find 
       $$ \| \partial_{t} \delta (t) \|_{\Hc^{s-4}_{\nu-2}}  \leq
        C(R)\langle t \rangle ^{3}h^2(1 + (\eps + \eps h^2) R_{2}^2) +   \langle t \rangle ^{3} \eps  h^4 R_{2}^2.$$ 
      This yields
 $$ I (t)  \leq   \left | \int_{0}^t  \langle D \mathcal{R}(\sigma), D \delta(\ssf(\sigma)) \rangle_{L^2} \, \dd\sigma \right |
     +  \langle t \rangle^{6} C(R)^{2}h^4( 1 +  (\eps + \eps h^2) R_{2}^2)^{2}.
  $$
    As in the previous proof, we write
     $$ \mathcal{R}(t,x,v)= \mathcal{R}^{1}(t,x,v) + \mathcal{R}^2(t,x,v) $$
      where $\mathcal{R}^1$ is given in Fourier by 
    \begin{multline*}
    \hat{R}^1_{n}(t,\xi)= 
        (t- \ssf(t)) \Big( np_{n}\big( \partial_{\xi}\hat \gs_{n}(\ssf(t), n \ssf(t)) \hat \eta(\xi - n\ssf(t)) + \hat  \gs_{n}(\ssf(t), n \ssf(t))  \partial_{\xi} \hat \eta(\xi - n\ssf(t))
\big)(n \ssf  (t) - \xi )
 \\ +
 \eps \sum_{k} k p_{k} ( \partial_{\xi} \hat \gs_{k}(\ssf(t), n \ssf(t))  \hat \gs_{n-k}(t, \xi- k \ssf(t)) +   \hat \gs_{k}(\ssf(t), n \ssf(t))  \partial_{\xi}\hat \gs_{n-k}(\ssf(t), \xi- k \ssf(t))) (n \ssf (t) - \xi) \Big)
     \end{multline*}
  and the remainder $\mathcal{R}^2$ is such that 
 $$
  \| \mathcal {R}^2 (t) \|_{\Hc^{s-3}_{\nu-2}}
   \leq C  h^2\Big( m^{(2)}_{t,s-2}(\gs) \big( \| \eta \|_{\Hc^1_{\nu}} +  \eps \|\gs \|_{\Hc^1_{\nu}} \big)+  m^{(2)}_{t,2}(\gs)  \big( \| \eta \|_{\Hc^{s-2}_{\nu}}
    +   \eps  \| \gs \|_{\Hc^{s-2}_{\nu}}\big) \Big)
 $$
     with 
     $$  m^{(2)}_{t, \gamma} (h) =\langle t \rangle^{\gamma} \sup_{n= \pm1} \sup_{ \xi \in [nt, n\ssf(t)]} \sum_{\alpha + \beta \leq 2, \, \alpha \neq 2} |\partial_{t}^\alpha \partial_{\xi}^\beta\hat h_{n} ( \xi)| .$$
     Since $\nu>5/2$, we have that 
     $$  m^{(2)}_{t,s-2}(\gs) \leq C \langle t \rangle^2 \| \gs \|_{\Hc^{s-4}_{\nu}} +  \langle t \rangle^{2}  \| \partial_{ t}\gs \|_{\Hc^{s-4}_{\nu}}.$$
      and hence thanks to the equation \eqref{julienclerc} and  Theorem \ref{maintheod}, we obtain that
      $$ m^{(2)}_{t,s-2}(\gs)  C(R)  \langle t \rangle^2$$
      This yields
     $$  \| \mathcal {R}^2 (t) \|_{\Hc^{s-3}_{\nu-2}}  \leq C(R)  \langle t \rangle^2$$
     and hence 
    \begin{equation}
    \label{strangII}
    I (t)  \leq   \left | \int_{0}^t  \langle D \mathcal{R}^1, D \delta(\ssf(\sigma)) \rangle_{L^2} \, \dd\sigma \right |
     +  C(R)^{2} \langle t \rangle^{6} h^4( 1 +  (\eps + \eps h^2) R_{2}^2)^{2} + C(R)\langle t \rangle^{3} h^2  \sup_{[0, t]} \| \delta \|_{\Hc^{s-3}_{\nu-2}}.
     \end{equation}
      Thanks to the definition of $\mathcal{R}^1$,   assuming that $Nh \leq t \leq (N+1)h$, we obtain that
      $$  \left | \int_{0}^t  \langle D \mathcal{R}^1(\sigma), D \delta(\ssf(\sigma)) \rangle_{L^2} \, \dd\sigma \right | 
      =   \left | \int_{Nh}^t  \langle D \mathcal{R}^1(\sigma), D \delta(\ssf(\sigma)) \rangle_{L^2} \, \dd\sigma \right |
       \leq C(R) \langle t \rangle^{3} h^2  \sup_{[0, t]} \| \delta \|_{\Hc^{s-3}_{\nu-2}}.$$
     Once more we have used the cancellation argument (see lemma \ref{midpoint}).
       Consequently, from \eqref{strangfin2}, \eqref{strangII} and the above estimate, we obtain that
      \begin{multline*} \sup_{[0, T]} \| \delta \|_{\Hc^{s-3}_{\nu-2}}^2 \leq 
        C(R)^{2} \langle t \rangle ^{6} h^4(  1 +  (\eps + \eps h^2) R_{2}^2)^{2} \\
        + C(R) \langle t \rangle^{3} h^2( 1  +  (\eps + \eps h^2)  R_{2}^2) ( 1 + h^2R_{2}) \sup_{[0, T]} \| \delta \|_{\Hc^{s-3}_{\nu-2}}.
        \end{multline*}
         By using the Young inequality this yields
         $$   \sup_{t\in [0, T]} \langle t \rangle^{-3} \| \delta \|_{\Hc^{s-3}_{\nu-2}}
          \leq    C(R)h^2( 1 +  (\eps + \eps h^2) R_{2}^2) ( 1 + h^2 R_{2}).
         $$
           This ends the  proof.

\end{proof}

\subsection{Proof of proposition \ref{ordre2}}

Using Propositions \ref{propzetaz2}, \ref{propgsg-7} and \ref{propgsg2}, we have proven that 
$$ Q_{T, s-3, \nu-2}(\delta) \leq   C(R)h^2(  1 +  (\eps + \eps h^2) R_{2}^2) ( 1 + h^2 R_{2})$$
 if  $Q_{T, s-3, \nu-2}(\delta) \leq h^2 R_{2}$.
 We can thus obtain 
 Proposition \ref{ordre2} with the same bootstrap argument as before by choosing $R_{2}>C(R)$ and
  then  by taking $\eps$ and $h$ sufficiently small.

\section*{appendix}

 \subsection*{Proof of Lemma \ref{lemvolterra}} Let us first note that the equation \eqref{beethov} only involves $K(n,t)$ for positive $t$, hence $K(n,t)$ can be replaced here by $K_1(n,t)$ (see \eqref{kernel}). Let us  take $T>0$, and let us set for the purpose of the proof $K(t)= K_1(n,t)$, 
  $ F(t) = ( F_{n}(t) + G_n(t))  \mathds{1}_{ 0 \leq t  \leq T}. $ Since we  only consider the cases $n= \pm 1$, we do not write down anymore explicitly the dependence in $n$.  We consider the equation
 \begin{equation}
 \label{eqvolterra2}
   y(t)= K *  y (t)  +  F(t), \quad t \in \mathbb{R}
 \end{equation}
setting $y(t) = 0$ for $t \leq 0$. 
 Note that the solution of this equation coincides with $ \zeta_n(t)$  on $[0, T]$ since the modification of the source term for $t \geq T$ does not affect the past. 
     By taking the Fourier transform in $t$ (that we still denote by $\hat \cdot$\, ), we obtain
 \begin{equation}
 \label{voltfourier}  \hat y (\tau)=  \hat{K}(\tau) \hat y (\tau) + \hat F(\tau), \quad \tau \in \R, 
 \end{equation}
with  $ \hat{K}(\tau)= \hat{K}(n, \tau)$. Under the assumption ${\bf (H)}$, the solution of \eqref{voltfourier} is given explicitely by the formula
   \begin{equation}
   \label{vol2} \hat{y}(\tau)= { \hat{F}(\tau) \over 1 - \hat{K}(\tau)}.
   \end{equation} 
  Let us observe that  since $( 1 + v^2) \eta_{0} \in \Hc^5$,   we have by \eqref{emb1} that for $\alpha \leq 2$ and  for $t > 0$
    \begin{equation}
    \label{decvol2}| \partial_{t}^\alpha  K(t)   | \leq {C \over \langle t \rangle^4}    \in L^1(\mathbb{R}_{+}).
    \end{equation}
    Note that by definition of $K(t)$, the function $K(t)$ is continuous in $t = 0$, but not $C^1$. 
     Using an integration by parts on the definition of the Fourier transform, we  then get 
      that
     \begin{equation}
    \label{vol1}
      |  \partial_{\tau}^\alpha \hat{K}(\tau) | \leq { C \over \langle \tau \rangle^2}, \quad  \alpha  \leq 2.
      \end{equation}
 To get this, we  have used that the function $t\, \hat \eta_{0}(t)$ vanishes at zero. 
 
  By using this estimate on $\hat{K}$,  {\bf (H)}  and that  $\hat F(\tau) \in H_\tau^{1}$ (the Sobolev space in $\tau$) since $F$ is compactly supported in time, we easily
   get that $ y$ defined via its Fourier transform  by  \eqref{vol2}  belongs to $H_{\tau}^1$. This  implies that $\langle t \rangle y \in L^2$ and thus that $y \in L^1_{t}$.
  These remarks justify the  use  of the Fourier transform and that the function $y$ defined through its Fourier transform via \eqref{vol2}
   is a solution of  \eqref{eqvolterra2}. Moreover, thanks to \eqref{vol2} and {\bf (H)}, we  get that $\hat y$ can be continued 
    as an holomorphic function  in $\mathrm{Im}\, \tau \leq 0$ and  thanks to a Paley Wiener type  argument, that $y$ vanishes for $t \leq 0$.  We have thus obtained  an $L^1$ solution
     of \eqref{eqvolterra2} that vanishes  for $t \leq 0$. By a Gronwall type argument, we easily get  that there is a unique solution in this class  of \eqref{eqvolterra2} and thus
     we have obtained the expression of the unique solution.    
 
  We can thus now focus on the proof of the estimate stated in Lemma \ref{lemvolterra}.
      Note that  a $L^2$-based version of this estimate would be very easily obtained.   %Note that $ 1 - \hat{K}(n,\tau + i \rho)$ does not vanish thanks to ${\bf 
The difficulty here is  to get the uniform   $L^\infty$ in  time  estimate we want to prove. 

\medskip 

   We shall first prove the estimate for $\gamma =0$.
Let us  take $\chi(\tau) \in [0, 1]$ a smooth compactly supported function  that  vanishes for $| \tau | \geq 1$ and which is equal to one for $|\tau| \leq 1/2$.   We  define $\chi_{R}(\tau)= \chi(\tau/R)$ and $\chi_R(\partial_t)$ the corresponding operator in $t$ variable corresponding to the convolution with the inverse Fourier transform of $\chi_R(\tau)$. 
  Thanks to \eqref{vol1}, we have that for $R$ large
 $$ \langle t \rangle^2 | (1- \chi_{R}(\partial_{t})) K (t) |  \leq C  \sum_{\alpha \leq 2}\|  \partial_{\tau}^\alpha ( ( 1 - \chi_{R}(\tau)) \hat{K}(\tau)) \|_{L^1(\R)}
  \leq C \int_{|\tau| \geq R/2} { 1 \over \langle \tau \rangle^2} \leq  { C \over R} $$
  and hence
  \begin{equation}
  \label{vol2'}
   \|  (1- \chi_{R}(\partial_{t} ) ) K(t) \|_{L^1(\R)} \leq {C \over R} \leq { 1 \over 2}
  \end{equation}
   for $R$ sufficiently large. This choice fixes $R$.
   
   To estimate the solution $y$ of \eqref{eqvolterra2}, we shall write that
    $$y=  \chi_{2R} (\partial_{t}) y + ( 1 - \chi_{2R}(\partial_{t}))y   = : y^l + y^h.$$
    By applying $ ( 1- \chi_{2R}(\partial_t)) $ to \eqref{eqvolterra2}, we get that
    $$
     y^h = K * y^h + ( 1 - \chi_{2R}(\partial_{t}) )  F =\big( ( 1 - \chi_{R}(\partial_{t}) K \big) * y^h +   ( 1 - \chi_{2R}(\partial_{t})  )F
    $$
     since $(1- \chi_{R})= 1$ on  the support of $1- \chi_{2R}$. Therefore, we obtain thanks to \eqref{vol2'} and the fact that $\chi_{2R}(\partial_t)$ is a convolution operator with a $L^1$ function,  that
     $$ \|y^h \|_{L^\infty} \leq  { 1 \over 2} \| y^h \|_{L^\infty} +  C \| F \|_{L^\infty}$$
      and hence
      $$ \|y^h \|_{L^\infty} \leq   2 C \|F\|_{L^\infty}.$$
       For the low frequencies, we can use directly the form \eqref{eqvolterra2} of the equation: We can write that
       $$  \hat y^l(\tau)= { \chi_{2R}(\tau)   \over 1 - \hat{K}(\tau)} \chi_{R}(\tau) \hat F(\tau).$$
        Since the denominator does not vanish thanks to  ${\bf(H) }$, we obtain again  that $y^l$ can be written as the convolution of an $L^1$ function - which is the inverse Fourier transform of $ \chi_{2R}(\tau)/( 1-  \hat{K}(\tau))$ -  by the function
         $\chi_{R}(\partial_{t}) F$ which is a convolution of $F$ by a smooth function. Thus we obtain by using again the Young inequality  that
         $$  \| y^l \|_{L^\infty} \leq  C \|F\|_{L^\infty}.$$
          Since 
          $ \|y \|_{L^\infty} \leq \| y^l \|_{L^\infty} + \|y^h\|_{L^\infty},$
           we get the  desired estimate  for $\gamma=0$.    
    To get the estimate for arbitrary $\gamma$, we can proceed by induction.  We observe that
    $$ t y(t) =  K * (t y) +  F^1$$
     with $F^1= (tK) * y + t  F$. 
     Using the result $\gamma = 0$, we obtain that $\| t  y \|_{L^\infty} \leq C \|F^1 \|_{L^\infty}$. Now 
      since $\eta_{0} \in \Hc^{\gamma + 3}$,  for $\gamma =1$, we obtain that $ tK  \in L^1$ and thus 
      $$  \|F^1 \|_{L^\infty} \leq C\big( \| t F\|_{L^\infty} +  \|y\|_{L^\infty}) \leq C \| ( 1 + t ) F\|_{L^\infty}.$$
       The higher order estimates follow easily in the same way.     
%\end{proof}

\subsection*{Proof of Lemma \ref{lemcom}}
 We give the proof of \eqref{com1}, the proof of the second estimate  being  slightly easier. In the Fourier side, we have   for $\mathcal{L}_{\sigma}[\gs] (h)$ the expression
$$
 (  \mathcal{F}{\Lc_\sigma[\gs] h})_n(\xi) = \sum_{k \in \{ \pm 1\} }  kp_k  \zs_k(\sigma)  \hat h_{n-k}(\sigma,\xi - k\ssf(\sigma)) (  n  \ssf(\sigma) -  \xi).
$$
  Consequently, we obtain that
  \begin{multline*} \big(  \mathcal{F}  ( [D^{r,p,q},  {\Lc_\sigma[\gs]  }]h ) \big))_n(\xi) = \\
    \sum_{k \in \{ \pm 1\} }
      kp_k  \zs_k(\sigma) \Big( n^p \xi^q  \partial^r_\xi \big( \hat h_{n-k}(\sigma,\xi - k\ssf(\sigma)) (  n \ssf(\sigma)-  \xi) \big) -  \\
        \big( (n-k)^p  (\xi-  \ssf(\sigma))^q   \partial_{\xi}^r\hat h_{n-k} ( \sigma,\xi - k\ssf(\sigma)) (  n \ssf(\sigma) -  \xi) \big)  \Big).
       \end{multline*}
For $k =\pm 1$, we can thus expand the above expression into a finite sum of terms under the form
$$
 I_{n}^k(\sigma, \xi)=  k p_{k} \zs_{k} (\ssf(\sigma))  k^{p_{1}}  (n-k)^{ p - p_{1} + \alpha }  \big(k \ssf(\sigma) \big)^{q_{1}+ \alpha} (\xi - k\ssf(\sigma))^{q - q_{1} + \beta}
 \partial_{\xi}^{r_{1}} \hat h_{n-k}(\sigma,\xi - k\ssf(\sigma))
$$
where 
$$0 \leq   p_{1} \leq p, \,    0  \leq  q_{1} \leq q ,  \quad  m-1 \leq r_{1} \leq m, \quad   \alpha + \beta =  r_{1} - m  + 1, \, \alpha, \, \beta \geq 0.$$
Moreover, if $r_{1}=r$, then we have $p_{1}+ q_{1}>0$.

 We have to estimate
 $ \sum_{n} \int_{\xi}  |\sum_{k \in \pm 1}I_{n}^k (\sigma, \xi) |^2 \, d\xi$ by isometry of the Fourier transform.

 %Let us start with the case $m_{1}= m$. 
 We note that for a fixed $k \in \{\pm 1\}$ then for $|n - k| + | \xi - k \ssf(\sigma)| \leq |k| \ssf(\sigma)$, we have 
 $$ | I_{n}^k (\sigma, \xi) | \leq  C \ssf(\sigma)^{ p + q + 1 } |\zs_{k}(\ssf(\sigma))| |n-k|| \partial_{\xi}^{r_1} \hat h_{n-k} (\sigma, \xi - k \ssf(\sigma)) |$$
  whereas for $|n-k| + | \xi- k \ssf(\sigma) | \geq  |k| \ssf(\sigma)$, we have
  $$| I_{n}^k (\sigma, \xi) |  \leq  C  \langle \sigma \rangle^2 |\zs_{k}(\ssf(\sigma))| (|n-k| + | \xi- k \ssf(\sigma) |)^\gamma | \partial_{\xi}^{r_1} \hat h_{n-k} (\sigma, \xi - k \ssf(\sigma)) |. $$ 
 
  Consequently  by taking the $L^2$ norm, we find that
 $$ \| \sum_{k \in \pm 1}I_{n}^k(\sigma,\xi) \|_{L^2} \leq  C \big( m_{\sigma, \gamma + 1}(\zs)  \|h(\sigma) \|_{\Hc^1} +  m_{\sigma, 2}(\zs) \|h(\sigma)\|_{\Hc^m}\big).$$
 %  We proceed in the same way for the case  $m_{1}= m-1$ and find the same estimate.  Indeed, the only difference is that in the region
%  $|n -k| + | \xi- k \ssf(\sigma) | \geq  |k| \ssf(\sigma)$, the case $q_{1}= p_{1}= 0$ can occur, but since $\alpha = \beta = 0$, in this case, 
%   we readily  have 
%   $$ | I_{n}(\ssf(\sigma), \xi) | \leq C \zs_{k}(\ssf(\sigma))   (|n-k| + | \xi- k \ssf(\sigma) |)^\gamma | \partial_{\xi}^{m_{1}} \hat h (\ssf(\sigma), \xi - k \ssf(\sigma)) |$$
%    which is still a good estimate.
 This ends the proof of the  Lemma.

\end{document}